\documentclass[10pt]{article}
\usepackage[utf8]{inputenc}
\usepackage[T1]{fontenc}
\usepackage{amsmath,amssymb,amsthm}
\usepackage{geometry}
\usepackage{graphicx}
\usepackage{booktabs}
\usepackage{placeins}
\usepackage{url}
\usepackage{cite}
\usepackage{color}

\newtheorem{theorem}{Theorem}[section]
\newtheorem{proposition}[theorem]{Proposition}
\newtheorem{lemma}[theorem]{Lemma}
\newtheorem{corollary}[theorem]{Corollary}
\newtheorem{definition}[theorem]{Definition}
\theoremstyle{remark}
\newtheorem{remark}[theorem]{Remark}

\title{Consistency and Convergence of the Backward-Euler Scheme for Stochastic Functional Differential Equations Driven by Fractional Brownian Motion}

\author{
Alexander Abreu$^{1}$\thanks{Alexander Abreu acknowledges financial support through a doctoral fellowship
in Mathematics at PUCV, Chile. He also gratefully acknowledges the hospitality
of Ernesto Mordecki during his visit to Uruguay.}
\and
Lisandro Fermín$^{2}$\thanks{Lisandro Fermín acknowledges financial support from Aix-Marseille University.
The project leading to this publication has received funding from the French government under the “France 2030” investment plan managed by the French National Research Agency (reference :ANR-17-EURE-0020) and from Excellence Initiative of Aix-Marseille University - A*MIDEX.} \thanks{Corresponding author.}
\and
Ernesto Mordecki$^{3}$\thanks{Ernesto Mordecki was partially supported by MathAmSud EXPLORE-SDE
AMSUD240037}
\and
Soledad Torres$^{1}$\thanks{Soledad Torres was partially supported by FONDECYT Regular Projects
Nos.~1230807 and 1221373, MathAmSud EXPLORE-SDE AMSUD240037, and
MathAmSud SiJaVol AMSUD240024. This work was also supported by the
Centro de Modelamiento Matemático (CMM), through the BASAL fund
FB210005 from ANID-Chile.}
}

\date{}

\begin{document}
\maketitle

\begin{center}
\begin{minipage}{0.9\textwidth}
\small
\begin{itemize}
\item[$^{1}$] CIMFAV - Ingemat, Facultad de Ingeniería, Universidad de Valparaíso, Chile. \\
\texttt{pedro.abreu@uv.cl}\\
\texttt{soledad.torres@uv.cl}
\item[$^{2}$] \texttt{Aix-Marseille Univ, CNRS, AMSE, France.} \\
\texttt{lisandro.fermin@uni-amu.fr}
\item[$^{3}$] Centro de Matemática, Facultad de Ciencias, Universidad de la República, Montevideo, Uruguay. \\
\texttt{mordecki@cmat.edu.uy}
\end{itemize}
\end{minipage}
\end{center}

\begin{abstract}
We study the backward-Euler scheme for a class of stochastic functional
differential equations (SFDEs) with memory driven by a fractional
Brownian motion with Hurst parameter $H>1/2$. We first establish the local
consistency of the method, with a local truncation error of order
$H-\rho-\beta+1$, and then, as the main result, prove the uniform global
pathwise convergence of the numerical approximation, on a set of
probability one, with order $H-\rho-\beta$, for
$\beta\in(1-H,1/2)$ and $0<\rho<H-\beta$. A uniform
high-probability formulation with deterministic constants follows as a
corollary. In particular, the convergence order can be taken arbitrarily
close to $2H-1$, matching the order obtained for the explicit Euler scheme. The convergence proof
relies on an exact integral representation of the scheme at mesh
points, an a priori bound for the numerical solution, fractional
calculus estimates, and a Gronwall inequality for weakly singular
kernels.
\end{abstract}



\section{Introduction}
\label{sec:introduccion}
 
Stochastic functional differential equations (SFDEs) driven by
fractional Brownian motion (fBm) provide a natural framework for
modelling systems whose evolution depends on their own past through a
memory term, while the driving noise exhibits long-range dependence or
anti-persistence, according to the Hurst parameter $H$. When
$H>1/2$, the fBm has Hölder-continuous paths of order strictly larger
than $1/2$, which allows the stochastic integral to be defined
pathwise, in the generalized Riemann-Stieltjes sense of
Zähle~\cite{zahle1998}, without resorting to Itô calculus. This
pathwise viewpoint, developed in the work of
Nualart and Rășcanu~\cite{nualart2002} on differential equations driven
by Hölder-continuous signals, is the technical backbone of the entire
analysis developed below.
 
Existence and uniqueness theory for SDEs driven by fBm with $H>1/2$ is
by now well established, and it has been extended in several
directions relevant to the present setting. For delay and functional
equations, Ferrante and Rovira~\cite{ferrante2006} proved existence and
uniqueness for scalar delay equations driven by fBm and studied the
small-delay limit, while Boufoussi and Hajji~\cite{boufoussi2011}
established existence and uniqueness for a broader class of SFDEs with
a general memory functional, in a framework close to the one adopted
here. On the numerical side, Kloeden and
Neuenkirch~\cite{kloeden2007pathwise} obtained general pathwise
convergence results for approximation schemes of SDEs driven by
Hölder-continuous noise, including fBm, and Garrido-Atienza, Kloeden
and Neuenkirch~\cite{garrido2009} used discretization schemes of this
type to approximate stationary solutions of fBm-driven systems.
Neuenkirch and Nourdin~\cite{neuenkirch2007} identified the exact rate
of convergence of the explicit Euler scheme for SDEs driven by fBm with
$H>1/2$ in the memoryless case, showing it to be $2H-1$, a result
complemented by Mishura and Shevchenko~\cite{mishura2008}; higher-order
schemes, such as the Milstein-type scheme of Deya, Neuenkirch and
Tindel~\cite{deya2012}, have also been developed, although at the cost
of more restrictive regularity assumptions on the coefficients.
Implicit discretizations for fractional-noise equations have also been
studied in important settings without the distributed-memory structure
considered here. Kamrani and Jamshidi~\cite{kamrani2017} analysed an
implicit Euler approximation for stochastic evolution equations driven
by fBm. Zhang and Yuan~\cite{zhangyuan2021} developed an implicit Euler
scheme for one-dimensional fBm-driven SDEs with locally Lipschitz drift,
and Zhou, Hu and Liu~\cite{zhouhu2023} treated the backward Euler method
under a one-sided Lipschitz condition. These results show that
implicitness itself is not new in the fractional-noise literature; the
issue addressed here is its interaction with a distributed functional
memory term and with the pathwise fractional-calculus estimates needed
to control that term.

A recent contribution by Garzón, León, Lozada and
Torres~\cite{garzon2026euler} studies an Euler approximation for
stochastic functional differential equations driven by fractional
Brownian motion in a functional Hölder-space setting. In their
formulation, the coefficients act on past trajectory segments belonging
to a space of \(\lambda\)-Hölder continuous functions, with
\(\lambda\in(1/2,H)\), the stochastic integral is interpreted as a Young
integral, and rough-path techniques are used to establish a convergence
rate \(n^{-\gamma}\) for every \(\gamma<2\lambda-1\). The class studied
in the present paper is organized around the explicit distributed-memory
functional
$
Y(t)=\int_{-r}^{0}K(t,s,X(t+s))\,ds,
$
which is discretized together with the state equation. The numerical
analysis therefore involves, in addition to the state approximation,
the quadrature error of the distributed-memory term and its fractional
increments. In the backward-Euler scheme considered here, this
distributed-memory approximation is coupled with the implicit
right-endpoint evaluation of the drift, while the stochastic integral
is treated pathwise through generalized Lebesgue-Stieltjes fractional calculus.

In a previous work~\cite{abreu2026euler} we studied the convergence of
the explicit Euler scheme for the class of SFDEs with memory driven by
fBm with $H>1/2$ considered here. That analysis rested on three
pillars: the Hölder regularity of the exact solution, a careful
discretization of the memory functional, and the pathwise
interpretation of the stochastic integral in the sense of Zähle. These
results provide the technical foundation for the study of the implicit
scheme addressed in the present paper.
 
Implicit time-stepping methods are a well-established tool in the
numerical analysis of both deterministic and stochastic differential
equations - see, e.g., the classical treatment of implicit and
semi-implicit schemes in Kloeden and
Platen~\cite{kloedenplaten1992} - and are especially relevant for
stiff or dissipative drifts. In the delay literature driven by standard
Brownian motion, Buckwar~\cite{buckwar2004} analysed the
$\theta$-Maruyama family for SFDEs with a distributed memory term, so
that the backward choice already has a natural functional analogue in
the Itô setting. Combining this distributed-memory structure with fBm
is considerably more delicate: the non-semimartingale noise requires a
pathwise fractional-calculus treatment, the memory quadrature introduces
increments over the whole delay window, and the right-endpoint
evaluation of the drift creates an implicit step. The contribution of
the present paper is to establish a complete consistency and global
convergence analysis for this combination of features, with a
state- and memory-dependent diffusion coefficient and an explicit
control of the right-endpoint projection $C_h^+$.
 
The backward-Euler discretization retains the left-point approximation
of the diffusion coefficient and the same distributed-memory
quadrature as the explicit scheme of~\cite{abreu2026euler}, while
evaluating the drift at the right endpoint of each discretization
subinterval. This right-endpoint evaluation changes the recursive
structure of the numerical method and propagates through the
well-posedness, a priori, and global error analyses. Since the drift at
time \(t_{n+1}\) depends on the yet-unknown value
\(\widetilde X_h(t_{n+1})\), each numerical step requires an implicit
solve whose existence and uniqueness must first be established.
Moreover, the shift from the left-point projection \(C_h\) to the
right-point projection \(C_h^+\) introduces additional terms in the
drift estimates that must be controlled throughout the fractional
analysis. This distinction allows us to isolate the analytical
mechanisms that are specific to the backward discretization and to
derive estimates that close simultaneously for the state, the
distributed-memory functional, and their fractional increments.

The relevance of this analysis is twofold. On the theoretical side, it
extends the pathwise numerical analysis developed for the explicit
Euler scheme in~\cite{abreu2026euler} to a backward discretization of
fBm-driven SFDEs with distributed memory. The analysis combines the
well-posedness of the implicit recursion with uniform pathwise estimates
and a coupled control of the state error, its fractional increments,
and the error generated by the discretization of the memory functional. On the
modelling side, systems with mean-reverting or dissipative drift and
delayed feedback are natural candidates for such a discretization. The
present article establishes well-posedness of each implicit step and
convergence under the standing Lipschitz assumptions; stronger
stability results under monotonicity or one-sided Lipschitz conditions
constitute a separate extension and are not required for the
convergence theorem proved here.
 
It is worth emphasizing precisely what changes, and what does not,
relative to the explicit Euler scheme of \cite{abreu2026euler}. The two
schemes share the same discretization of the diffusion coefficient (at
the left endpoint $t_n$) and the same discretization of the memory
functional; consequently, every estimate that involves only the
diffusion transfers essentially unchanged, and, as shown in
Corollary~\ref{cor:2H-1}, both schemes attain the same convergence rate,
arbitrarily close to $2H-1$. The two schemes differ only in the
evaluation point of the drift coefficient - $t_n$ for the explicit
scheme, $t_{n+1}$ for the backward-Euler scheme - but this single
change has three consequences that are not present in the explicit
case and that account for most of the technical work in this paper: (i)
each step is now defined only implicitly, through a fixed-point
equation, so existence and uniqueness of $\widetilde X_h(t_{n+1})$ must
be established first (Proposition~\ref{prop:existencia-implicito}); (ii)
the drift term in the integral representation of the scheme, and in
every estimate that involves it, must be re-derived with the projection
operator $C_h^+$ in place of $C_h$; and (iii) the well-posedness
condition $hL_2<1$, absent in the explicit scheme, introduces an
additional step-size restriction associated specifically with the
implicit solve; the a priori argument also requires the usual smallness
conditions needed to absorb mesh-dependent remainder terms.
Remark~\ref{rem:explicit-backward} in Section~\ref{sec:4} makes this
comparison precise at the level of the integral representation of the
scheme.
 
The main contributions of this paper are the following. First, we
establish the local consistency of the backward-Euler scheme, obtaining
a local truncation error of order $O(h^{H-\rho-\beta+1})$
(Theorem~\ref{thm:4.1.2}). Second, we derive an exact integral
representation of the scheme at the mesh points together with uniform a
priori estimates for its auxiliary interpolation
(Proposition~\ref{prop:existencia-implicito} and
Theorem~\ref{thm:apriori-backward}). These estimates allow us, third, to
prove the uniform global convergence of the method with rate
$H-\rho-\beta$, for $\beta\in(1-H,1/2)$ and $0<\rho<H-\beta$
(Theorem~\ref{thm:convergencia-backward}). 
As a consequence, the convergence order can be chosen arbitrarily close
to the classical \(2H-1\) Euler rate established for memoryless
fractional Brownian motion driven SDEs~\cite{neuenkirch2007}
(Corollary~\ref{cor:2H-1}).
Throughout the proofs we systematically separate the
arguments that coincide with the analysis of the explicit Euler scheme
from those that depend specifically on the implicit evaluation of the
drift, so that the reader familiar with \cite{abreu2026euler} can follow
precisely where, and why, the two analyses diverge. A numerical study
complements the theoretical results and illustrates the practical
behaviour of both schemes on a concrete example.
 
The paper is organized as follows. Section~\ref{sec:2} presents the
stochastic functional differential equation, the standing hypotheses,
and the existence, uniqueness and regularity results used throughout
the analysis. Section~\ref{sec:3} introduces the backward-Euler scheme,
together with the mesh projection operators and the auxiliary
interpolation used in the estimates. Section~\ref{sec:4} contains the
consistency analysis, the a priori estimates, and the proof of
convergence. Finally, Section~\ref{sec:5} presents the numerical study.

%and Section~\ref{sec:conclusiones} collects the conclusions.

\section{The Stochastic Functional Differential Equation}
\label{sec:2}

We consider the stochastic functional differential equation
\begin{equation}\label{eq:X}
X(t) = \begin{cases}
X(0)+\displaystyle\int_0^t b(s,X(s),Y(s))\,ds+\int_0^t
\sigma(s,X(s),Y(s))\,dB^H(s), & t\in[0,T]\\[4pt]
\phi(t), & t\in[-r,0],
\end{cases}
\end{equation}
where $B^H=(B^H(t))_{t\ge0}$ is a fractional Brownian motion with Hurst
parameter $H>1/2$. The stochastic integral with respect to $B^H$ is
understood in the sense of the generalized Lebesgue-Stieltjes integral
introduced by Zähle~\cite{zahle1998}. The drift and diffusion
coefficients $b,\sigma:[0,T]\times\mathbb R\times\mathbb R\to\mathbb R$
are assumed to be continuous functions, and the deterministic initial
condition $\phi:[-r,0]\to\mathbb R$ is assumed to be Hölder continuous.
The memory effect is modelled through the functional $Y(t)$ defined by
\begin{equation}\label{eq:Y}
Y(t) = \int_{-r}^0 K(t,s,X(t+s))\,ds = \int_{t-r}^t K(t,s-t,X(s))\,ds,
\qquad t\ge0,
\end{equation}
where the kernel $K:[0,T]\times[-r,0]\times\mathbb R\to\mathbb R$ is a
continuous function.

We consider the following hypotheses on the coefficients $b,\sigma$,
the initial condition $\phi$, and the kernel $K$. We denote by
\begin{equation}\label{Cr}
C_r=C([-r,0];\mathbb R^2) \,\, \text{the space of continuous functions} \,\,
\xi=(\xi_1,\xi_2):[-r,0]\to\mathbb R^2, 
\end{equation}
endowed with the norm  
\(
\|\xi\|=\sup_{s\in[-r,0]}
\bigl(|\xi_1(s)|+|\xi_2(s)|\bigr).
\)

\medskip
\noindent\textbf{Hypothesis $\phi$.} The initial condition
$\phi:[-r,0]\to\mathbb R$ is deterministic and, for every $0<\rho<H$,
there exists a positive constant $L_\rho$ such that
\begin{equation}\label{Hip:phi}
|\phi(t)-\phi(s)| \le L_\rho|t-s|^{H-\rho},
\qquad s,t\in[-r,0].
\end{equation}

\noindent\textbf{Hypothesis 1.} There exist positive constants
$L_1,L_2,L_3$ such that the following properties hold.
\begin{align}\label{Hip:1}
|b(t,x,y)-b(s,x,y)|+|\sigma(t,x,y)-\sigma(s,x,y)| &\le L_1|t-s|. \nonumber \\
|b(t,x_1,y_1)-b(t,x_2,y_2)|+|\sigma(t,x_1,y_1)-\sigma(t,x_2,y_2)| &\le
L_2\big(|x_1-x_2|+|y_1-y_2|\big). \nonumber \\
|b(t,x,y)|+|\sigma(t,x,y)| &\le L_3\big(1+|x|+|y|\big).
\end{align}

\noindent\textbf{Hypothesis 2.} The function $\sigma(t,\xi)$ is
Fréchet-differentiable in the variable $\xi$. Moreover, there exist
positive constants $M_1,M_2,M_3$ such that, for all $\xi,\eta\in C_r$, defined in \eqref{Cr}
and $s,t\in[0,T]$,
\begin{align}\label{Hip:2}
|D_\xi\sigma(t,\xi)|_{L(C_r,\mathbb R)} &\le M_1. \notag\\
|D_\xi\sigma(t,\xi)-D_\xi\sigma(t,\eta)|_{L(C_r,\mathbb R)} &\le M_2\|\xi-\eta\|. \notag\\
|D_\xi\sigma(t,\xi)-D_\xi\sigma(s,\xi)|_{L(C_r,\mathbb R)} &\le M_3|t-s|.
\end{align}

Hypothesis~2 is kept in functional form because it is part of the
well-posedness framework for the SFDE. The numerical estimates use its
restriction to the two state variables appearing in \eqref{eq:X}. The
next lemma records explicitly that the regularity properties required
there are consequences of the functional hypothesis.

\begin{lemma}
\label{lem:sigma-restriction}
Let $J:\mathbb R^2\to C_r$ be the constant-path embedding
\[
(Jz)(\theta)=z,\qquad \theta\in[-r,0],\qquad z=(x,y),
\]
and endow $\mathbb R^2$ with $|z|_1=|x|+|y|$. Then $\|Jz\|=|z|_1$.
Define the pointwise coefficient used in \eqref{eq:X} by
\[
\sigma(t,z):=\sigma(t,Jz).
\]
Under Hypothesis~2, (\ref{Hip:2}) $z\mapsto\sigma(t,z)$ is differentiable and
\begin{align}
\label{eq:sigma-restriction-bounds}
\|D_z\sigma(t,z)\| &\le M_1,\notag\\
\|D_z\sigma(t,z)-D_z\sigma(t,z')\| &\le M_2|z-z'|_1,\notag\\
\|D_z\sigma(t,z)-D_z\sigma(s,z)\| &\le M_3|t-s|.
\end{align}
Thus every derivative estimate for $\sigma(t,x,y)$ used below follows
from Hypothesis~2 and is not an additional assumption on the numerical
scheme.
\end{lemma}
\begin{proof}
The map $J$ is linear and isometric from $(\mathbb R^2,|\cdot|_1)$ into
$C_r$. By the chain rule,
\[
D_z\sigma(t,z)=D_\xi\sigma(t,Jz)\circ J.
\]
Since $\|J\|=1$, the first estimate in
\eqref{eq:sigma-restriction-bounds} follows from the first line of
\eqref{Hip:2}. For $z,z'\in\mathbb R^2$,
\begin{align*}
\|D_z\sigma(t,z)-D_z\sigma(t,z')\|
&\le \|D_\xi\sigma(t,Jz)-D_\xi\sigma(t,Jz')\|\,\|J\|\\
&\le M_2\|Jz-Jz'\|=M_2|z-z'|_1,
\end{align*}
and the temporal estimate is obtained in the same way from the third
line of \eqref{Hip:2}.
\end{proof}

\begin{lemma}
\label{lem:sigma-second-increment}
Under Hypothesis~2, (\ref{Hip:2}) for $t_1,t_2\in[0,T]$ and
$z_i=(x_i,y_i)\in\mathbb R^2$, $i=1,\ldots,4$, we have
\begin{align*}
&|\sigma(t_1,z_1)-\sigma(t_1,z_2)
-\sigma(t_2,z_3)+\sigma(t_2,z_4)|\\
&\le M_1|z_1-z_2-z_3+z_4|_1\\
&\quad+\Bigl[M_3|t_1-t_2|
+M_2\bigl(|z_2-z_4|_1+|z_1-z_2-z_3+z_4|_1\bigr)\Bigr]
|z_1-z_2|_1.
\end{align*}
\end{lemma}
\begin{proof}
Set $a=z_1-z_2$ and $b=z_3-z_4$. By the fundamental theorem of
calculus along line segments,
\begin{align*}
&\sigma(t_1,z_1)-\sigma(t_1,z_2)
-\sigma(t_2,z_3)+\sigma(t_2,z_4)\\
&=\int_0^1\{D_z\sigma(t_1,z_2+\lambda a)a
-D_z\sigma(t_2,z_4+\lambda b)b\}\,d\lambda.
\end{align*}
Add and subtract
$D_z\sigma(t_2,z_4+\lambda b)a$ inside the integral. The term containing
$a-b$ is bounded by $M_1|a-b|_1$. For the remaining term,
Lemma~\ref{lem:sigma-restriction} gives
\begin{align*}
&\|D_z\sigma(t_1,z_2+\lambda a)
-D_z\sigma(t_2,z_4+\lambda b)\|\\
&\qquad\le M_3|t_1-t_2|
+M_2|z_2-z_4+\lambda(a-b)|_1\\
&\qquad\le M_3|t_1-t_2|
+M_2\bigl(|z_2-z_4|_1+|a-b|_1\bigr).
\end{align*}
Multiplication by $|a|_1$ and integration over $\lambda\in[0,1]$
prove the result.
\end{proof}

\noindent\textbf{Hypothesis 3.} For each $t\in[0,T]$, the map
\[
(s,x)\longmapsto K(t,s,x),
\qquad (s,x)\in[-r,0]\times\mathbb R,
\]
is continuously differentiable. Moreover, there exist positive
constants $K_1,K_2,K_3,K_4$ such that
\begin{align}\label{Hip:3}
|K(t,s,x)-K(u,v,x)| &\le K_1\big(|t-u|+|s-v|\big), \notag\\
|K(t,s,x)-K(t,s,y)| &\le K_2|x-y|, \notag\\
|K(t,s,x)| &\le K_3\big(1+|x|\big), \notag\\
\big|\nabla_{(s,x)}K(t,s,x)\big| &\le K_4, \notag\\
\big|\nabla_{(s,x)}K(t,s,x)-\nabla_{(s,x)}K(u,v,y)\big|
&\le K_4\big(|t-u|+|s-v|+|x-y|\big).
\end{align}

The differential regularity of Hypothesis~3 provides the four-point
estimate required in the analysis of the increments of the memory term.

\begin{lemma}
\label{lem:K-second-increment}
Under Hypothesis~3, let $t_1,t_2\in[0,T]$ and
$(s_i,x_i)\in[-r,0]\times\mathbb R$, $i=1,\ldots,4$. Then
\begin{align*}
&|K(t_1,s_1,x_1)-K(t_1,s_2,x_2)
-K(t_2,s_3,x_3)+K(t_2,s_4,x_4)|\\
&\le C\bigl(|s_1-s_2-s_3+s_4|+|x_1-x_2-x_3+x_4|\bigr)\\
&\quad+C\Bigl(|t_1-t_2|+|s_2-s_4|+|x_2-x_4|
+|s_1-s_2-s_3+s_4|+|x_1-x_2-x_3+x_4|\Bigr)\\
&\qquad\times\Bigl(|s_1-s_2|+|x_1-x_2|
+|s_3-s_4|+|x_3-x_4|\Bigr),
\end{align*}
where $C$ depends only on $K_4$.
\end{lemma}
\begin{proof}
The mean value theorem is applied to the map $(s,x)\mapsto K(t,s,x)$ on
the segments joining $(s_2,x_2)$ to $(s_1,x_1)$ and $(s_4,x_4)$ to
$(s_3,x_3)$. Subtracting both representations separates the second
increments of $(s,x)$ from the difference between the gradients. The
first contribution is controlled by the uniform bound on
$\nabla_{(s,x)}K$, while the second is estimated using the Lipschitz
condition of the last line of \eqref{Hip:3}. This gives the stated
bound.
\end{proof}

The following result collects the existence, uniqueness, and pathwise
regularity properties of the solution of \eqref{eq:X}. It follows from
the standard well-posedness theory for stochastic functional
differential equations driven by fractional Brownian motion with
$H>1/2$, together with the regularity assumptions imposed here on the
distributed-memory functional; see, in particular,
\cite{ferrante2006,boufoussi2011,nualart2002}. These properties were
also established for the present class of equations in
\cite{abreu2026euler}.

\begin{theorem}
\label{thm:4.1.1}
Under Hypotheses~$\phi$, 1, 2 and 3, equation \eqref{eq:X} admits a
unique solution. Moreover, for every $0<\rho<H$ there exists an almost
surely finite random variable $C(\omega)$ such that, almost surely,
\begin{equation}\label{Teo:EU}
|X(u)-X(v)| \le C(\omega)|u-v|^{H-\rho},
\qquad 0\le u,v\le T.
\end{equation}
\end{theorem}

\begin{proof}
For a continuous path segment \(\xi\in C([-r,0];\mathbb R)\), define
the distributed-memory functional

$$
\mathcal Y_t(\xi)
:=
\int_{-r}^{0}K(t,s,\xi(s))\,ds .
$$

Hypothesis~3 implies that this functional is Lipschitz continuous in
the path variable. Indeed,

$$
|\mathcal Y_t(\xi)-\mathcal Y_t(\eta)|
\le
K_2\int_{-r}^{0}|\xi(s)-\eta(s)|\,ds
\le
rK_2\|\xi-\eta\|_\infty ,
$$

and the growth bound on \(K\) gives the corresponding linear-growth
estimate for \(\mathcal Y_t\).

Consequently, after composing \(b\) and \(\sigma\) with the
distributed-memory functional, Hypotheses~1-3 and Hypothesis~\(\phi\)
place equation \eqref{eq:X} within the standard pathwise
well-posedness framework for functional differential equations driven
by fractional Brownian motion with \(H>1/2\); see
\cite{ferrante2006,boufoussi2011,nualart2002}. This yields existence
and uniqueness of the solution.

The pathwise fractional-integral estimates in the same framework,
together with the \(H-\rho\) Hölder regularity of the fractional
Brownian trajectories for every \(\rho>0\), give

$$
|X(u)-X(v)|
\le C(\omega)|u-v|^{H-\rho},
\qquad 0\le u,v\le T,
$$

for every \(0<\rho<H\), with \(C(\omega)<\infty\) almost surely.
This is \eqref{Teo:EU}.
\end{proof}

\begin{remark}
Since $X=\phi$ on $[-r,0]$, estimate \eqref{Teo:EU} together with
Hypothesis~$\phi$ implies that the extended path satisfies
\[
X\in C^{H-\rho}([-r,T])
\qquad\text{a.s.}
\]
Indeed, it only remains to consider $s\le0\le t$, in which case
\[
|X(t)-X(s)|
\le |X(t)-X(0)|+|\phi(0)-\phi(s)|
\le C(\omega)|t-s|^{H-\rho}.
\]
\end{remark}

\section{Backward-Euler Scheme for the SFDE}
\label{sec:3}

In this section we introduce the backward-Euler approximation of the
SFDE and describe how the fractional noise and the memory term are
discretized.

To simplify the notation, throughout the numerical analysis we assume
$T=r$. Given $N\in\mathbb N$, we set
\[
h:=\frac{T}{N}=\frac{r}{N},
\qquad
t_n=nh, \qquad n=-N,\dots,N.
\]
For the numerical estimates that follow, we fix
\(
\beta\in(1-H,1/2), \,\, 0<\rho<H-\beta,
\)
so that $H-\rho>\beta$ and, in particular, $H-\rho-\beta>0$. This
additional restriction on $\rho$ is used only in the numerical
analysis; Hypothesis~$\phi$ and Theorem~\ref{thm:4.1.1} retain their
original formulation for $0<\rho<H$.
We introduce the mesh projection operators. For $u\in[-r,T]$, we define
\begin{equation}\label{Chu}
C_h(u) = \Big\lfloor\frac{u}{h}\Big\rfloor h,
\end{equation}
and the projection operator onto the next mesh point by
\begin{equation}\label{Chu+}
C_h^+(u):=
\begin{cases}
C_h(u)+h, & u\in[-r,T),\\
T, & u=T.
\end{cases}
\end{equation}
Thus, if $u\in[t_i,t_{i+1})$, then $C_h(u)=t_i$ and $C_h^+(u)=t_{i+1}$.
The convention $C_h^+(T)=T$ avoids artificially introducing the point
$T+h$. The operator $C_h$ selects the left endpoint of the mesh
interval containing $u$, whereas $C_h^+$ selects its right endpoint.
This distinction will allow us to express, in a unified way, the
explicit evaluation of the diffusion and the implicit evaluation of the
drift.

The discrete delay functional is given by
\begin{equation}\label{Ytilde}
\widetilde Y(t) = \sum_{i=-N}^{-1} h\,K\big(t,t_i,\widetilde
X_h(t+t_i)\big).
\end{equation}
Equivalently,
\[
\widetilde Y(t) = \sum_{i=-N}^{-1}\int_{t_i}^{t_{i+1}}
K\big(t,C_h(s),\widetilde X_h(t+C_h(s))\big)\,ds
= \int_{-r}^0 K\big(t,C_h(s),\widetilde X_h(t+C_h(s))\big)\,ds.
\]
This representation makes explicit the link between the exact delay
term and its discrete counterpart. In particular, when $t=t_{n+1}$, the
largest argument of $\widetilde X_h$ appearing in \eqref{Ytilde} is
$t_n$; hence $\widetilde Y(t_{n+1})$ depends only on already computed
nodal values.

We fix the discrete initial condition by
\[
\widetilde X_h(t_n)=\phi(t_n),\qquad n=-N,\dots,0.
\]
For the estimates on continuous intervals we further extend the
approximation over the initial interval by
\[
\widetilde X_h(t)=\phi(t),\qquad t\in[-r,0].
\]

The backward-Euler scheme associated with \eqref{eq:X} is defined
recursively, for $n=0,\dots,N-1$, by
\begin{equation}\label{Backward}
\widetilde X_h(t_{n+1}) = \widetilde X_h(t_n) + h\,b\big(t_{n+1},
\widetilde X_h(t_{n+1}),\widetilde Y(t_{n+1})\big) +
\sigma\big(t_n,\widetilde X_h(t_n),\widetilde Y(t_n)\big)\,\Delta B^H(t_n),
\end{equation}
where
\[
\Delta B^H(t_n) = B^H(t_{n+1})-B^H(t_n).
\]
The implicitness of the method appears only in the drift variable
$\widetilde X_h(t_{n+1})$. Indeed, $\widetilde Y(t_{n+1})$ is determined
by previous values of the scheme, while the diffusion is evaluated at
the left endpoint $t_n$.

For the analytical estimates we introduce the continuous extension, for
$t\in[t_n,t_{n+1})$,
\[
\widetilde X_h(t) = \widetilde X_h(t_n) + b\big(t_{n+1},\widetilde
X_h(t_{n+1}),\widetilde Y(t_{n+1})\big)(t-t_n) +
\sigma\big(t_n,\widetilde X_h(t_n),\widetilde Y(t_n)\big)
\big(B^H(t)-B^H(t_n)\big).
\]
The nodal values $\widetilde X_h(t_n)$ constitute the numerical scheme.
For $t\in(t_n,t_{n+1})$, however, this extension contains the future
value $\widetilde X_h(t_{n+1})$ and is therefore, in general, not
$\mathcal F_t$-measurable. We use it only as an auxiliary pathwise
interpolation to establish estimates between mesh points. Although this
extension is not, in general, adapted in the interior of each interval,
the main convergence result will be formulated pathwise in the uniform
norm on $[-r,T]$.

Consider the auxiliary process $\widetilde X_h^{\,n+1}$: it is the value
obtained through scheme \eqref{Backward} when the exact value of the
solution is used at step $n+1$, that is,
\begin{equation}\label{eq:Xtilde}
\widetilde X_h^{\,n+1} := X(t_n) + h\,b\big(t_{n+1},X(t_{n+1}),\widehat
Y(t_{n+1})\big) + \sigma\big(t_n,X(t_n),\widehat Y(t_n)\big)\,\Delta B^H(t_n).
\end{equation}
Here, $\widehat Y(t_m)$ denotes formula \eqref{Ytilde} when
$\widetilde X_h(t+t_i)$ is replaced by $X(t+t_i)$, that is,
\begin{equation}\label{eq:Yhat}
\widehat Y(t) = \sum_{i=-N}^{-1} h\,K(t,t_i,X(t+t_i)).
\end{equation}
Equivalently,
\[
\widehat Y(t) = \sum_{i=-N}^{-1}\int_{t_i}^{t_{i+1}}
K\big(t,C_h(s),X(t+C_h(s))\big)\,ds = \int_{-r}^0
K\big(t,C_h(s),X(t+C_h(s))\big)\,ds.
\]

\begin{definition}\label{localerror}
The local error is defined by $\delta_{n+1}:=|X(t_{n+1})-\widetilde
X_h^{\,n+1}|$.
\end{definition}

\begin{definition}\label{cnsistent}
The scheme is said to be consistent of order $p$ if
$\delta_{n+1}\le Ch^{p+1}$, as $h\to0$.
\end{definition}

In the next section we introduce some technical lemmas needed to obtain
the consistency of the backward-Euler scheme.

\subsection{Estimates for the Numerical Scheme}
\label{sec:4.4}

The following lemma tells us how close the approximation of the memory
term $\widehat Y$ is to the memory term $Y$.

\begin{lemma}
\label{lem:4.4.1}
Let $\widehat Y$ be the approximation defined by \eqref{eq:Yhat}. Then,
for every $u\in[0,T]$, the following estimates hold:
\begin{equation}\label{lem:Y-Yhat}
\begin{aligned}
|Y(u)-\widehat Y(C_h(u))| &\le C(\omega)h^{H-\rho},\\
|Y(u)-\widehat Y(C_h^+(u))| &\le C(\omega)h^{H-\rho}.
\end{aligned}
\end{equation}
where $C(\omega)$ does not depend on $h$.
\end{lemma}
\begin{proof}
By definitions \eqref{Chu}-\eqref{Chu+},
\[
|u-C_h(u)|\le h 
\,\,\,\, \text{and} \,\,\,\,
|u-C_h^+(u)|\le h,
\qquad u\in[0,T].
\]
Consider first $C_h(u)$. Using the integral representation of
$\widehat Y$ given in \eqref{eq:Yhat}, Hypothesis~3, and the
$H-\rho$-Hölder regularity of the extended path of $X$ on $[-r,T]$,
obtained from Hypothesis~$\phi$ and Theorem~\ref{thm:4.1.1}, we have
\begin{align*}
|Y(u)-\widehat Y(C_h(u))|
&= \Big|\int_{-r}^0
\Big[K(u,s,X(u+s))
      -K\big(C_h(u),C_h(s),X(C_h(u)+C_h(s))\big)\Big]ds\Big|\\
&\le \int_{-r}^0
\Big(K_1|u-C_h(u)|+K_1|s-C_h(s)|\\
&\hspace{4.5cm}
+K_2|X(u+s)-X(C_h(u)+C_h(s))|\Big)ds.
\end{align*}
Moreover,
\[
|(u+s)-(C_h(u)+C_h(s))|
\le |u-C_h(u)|+|s-C_h(s)|\le 2h,
\]
so that
\[
|X(u+s)-X(C_h(u)+C_h(s))|
\le C(\omega)h^{H-\rho}.
\]
Since $0<h\le1$ and $H-\rho<1$, we obtain
\[
|Y(u)-\widehat Y(C_h(u))|
\le C(\omega)h^{H-\rho}.
\]
The second estimate is obtained by replacing $C_h(u)$ with $C_h^+(u)$
in the calculation above. Indeed, $|u-C_h^+(u)|\le h$ and
\[
|(u+s)-(C_h^+(u)+C_h(s))|\le 2h,
\]
so the same argument gives
\[
|Y(u)-\widehat Y(C_h^+(u))|
\le C(\omega)h^{H-\rho}.
\]
\end{proof}

We now show that the difference between the approximations
$\widetilde Y$ and $\widehat Y$ depends on the difference between the
solution of the SFDE and the backward-Euler scheme.

\begin{lemma}
\label{lem:4.4.2}
Let $\widetilde Y$ and $\widehat Y$ be the approximations defined by
\eqref{Ytilde} and \eqref{eq:Yhat}, respectively. Then, for each
$n=0,1,\dots,N-1$,
\begin{equation}\label{lem:Yhat-Ytilde}
|\widehat Y(t_{n+1})-\widetilde Y(t_{n+1})|
\le K_2 h\sum_{j=0}^n
|X(t_j)-\widetilde X_h(t_j)|.
\end{equation}
\end{lemma}
\begin{proof}
By Hypothesis~3,
\begin{align*}
|\widehat Y(t_{n+1})-\widetilde Y(t_{n+1})|
&= \Big|\sum_{i=-N}^{-1}h\,
\Big[K(t_{n+1},t_i,X(t_{n+1}+t_i))\\
&\hspace{3.4cm}
-K(t_{n+1},t_i,\widetilde X_h(t_{n+1}+t_i))\Big]\Big|\\
&\le K_2h\sum_{i=-N}^{-1}
|X(t_{n+i+1})-\widetilde X_h(t_{n+i+1})|.
\end{align*}
If $i\le -(n+2)$, then $n+i+1<0$ and, by the initial condition,
\[
X(t_{n+i+1})=\widetilde X_h(t_{n+i+1})=\phi(t_{n+i+1}),
\]
so all such terms vanish. Hence, only the indices
$i=-(n+1),\ldots,-1$ remain, and
\[
|\widehat Y(t_{n+1})-\widetilde Y(t_{n+1})|
\le K_2h\sum_{i=-(n+1)}^{-1}
|X(t_{n+i+1})-\widetilde X_h(t_{n+i+1})|.
\]
Changing the index to $j=n+i+1$, for which $j=0,\ldots,n$, we obtain
\[
|\widehat Y(t_{n+1})-\widetilde Y(t_{n+1})|
\le K_2h\sum_{j=0}^{n}
|X(t_j)-\widetilde X_h(t_j)|,
\]
which is \eqref{lem:Yhat-Ytilde}.
\end{proof}

\section{Consistency and Convergence}\label{sec:4}

\subsection{Local Error of the Backward-Euler Method}
\label{sec:consistencia-backward}

We first establish the local consistency of the method.

\begin{theorem}
\label{thm:4.1.2}
Suppose that Hypotheses~$\phi$, 1, 2 and 3 hold and that the parameters
$\beta$ and $\rho$ satisfy
\[
\beta\in(1-H,1/2),\qquad 0<\rho<H-\beta.
\]
Then the local truncation error (Definition \ref{localerror}) of the backward-Euler method
\eqref{Backward} defined in \eqref{eq:Xtilde} satisfies, uniformly on the mesh,
\begin{equation}\label{Teo:delta}
\max_{0\le n\le N-1}\delta_{n+1}
\le C(\omega)h^{H-\rho-\beta+1},
\qquad h\to0.
\end{equation}
where $C(\omega)$ does not depend on $h$ or $n$. Consequently, according
to the previous definition, the method is consistent of order
$H-\rho-\beta$.
\end{theorem}
\begin{proof}
By the exact equation \eqref{eq:X} and the definition of the auxiliary
process \eqref{eq:Xtilde},
\begin{align*}
X(t_{n+1})-\widetilde X_h^{\,n+1}
&= \int_{t_n}^{t_{n+1}}
\Big[b(s,X(s),Y(s))
-b\big(C_h^+(s),X(C_h^+(s)),\widehat Y(C_h^+(s))\big)\Big]ds\\
&\quad +\int_{t_n}^{t_{n+1}}
\Big[\sigma(s,X(s),Y(s))
-\sigma\big(C_h(s),X(C_h(s)),\widehat Y(C_h(s))\big)\Big]dB^H(s).
\end{align*}
Indeed, for $s\in[t_n,t_{n+1})$ we have $C_h(s)=t_n$ and
$C_h^+(s)=t_{n+1}$; the values at the endpoint $t_{n+1}$ do not affect
the integrals.

Applying the fractional Young-type estimate to the integral with
respect to $B^H$, we obtain
\[
\delta_{n+1}
=|X(t_{n+1})-\widetilde X_h^{\,n+1}|
\le I_1+I_2+I_3,
\]
where
\begin{align*}
I_1&:=\int_{t_n}^{t_{n+1}}
\Big|b(s,X(s),Y(s))
-b\big(C_h^+(s),X(C_h^+(s)),\widehat Y(C_h^+(s))\big)\Big|ds,\\
I_2&:=C(\omega)\int_{t_n}^{t_{n+1}}
\frac{\big|\sigma(s,X(s),Y(s))
-\sigma(C_h(s),X(C_h(s)),\widehat Y(C_h(s)))\big|}
{(s-t_n)^\beta}\,ds,\\
I_3&:=C(\omega)\int_{t_n}^{t_{n+1}}\int_{t_n}^{s}
\frac{|G(s)-G(u)|}{(s-u)^{\beta+1}}\,du\,ds,
\end{align*}
with
\[
G(v):=\sigma(v,X(v),Y(v))
-\sigma\big(C_h(v),X(C_h(v)),\widehat Y(C_h(v))\big).
\]
For $I_1$, by Hypothesis~1, Theorem~\ref{thm:4.1.1}, and the second
estimate of Lemma~\ref{lem:4.4.1},
\begin{align*}
I_1
&\le \int_{t_n}^{t_{n+1}}
\Big(L_1|s-C_h^+(s)|
+L_2|X(s)-X(C_h^+(s))|\\
&\hspace{4.2cm}
+L_2|Y(s)-\widehat Y(C_h^+(s))|\Big)ds\\
&\le C(\omega)\int_{t_n}^{t_{n+1}}
\big(h+h^{H-\rho}\big)ds
\le C(\omega)h^{H-\rho+1}.
\end{align*}
Here we used $0<h\le1$ and $H-\rho<1$, so $h\le h^{H-\rho}$.

For $I_2$, by Hypothesis~1, Theorem~\ref{thm:4.1.1}, and the first
estimate of Lemma~\ref{lem:4.4.1},
\begin{align*}
I_2
&\le C(\omega)\int_{t_n}^{t_{n+1}}
\frac{L_1|s-C_h(s)|
+L_2|X(s)-X(C_h(s))|
+L_2|Y(s)-\widehat Y(C_h(s))|}
{(s-t_n)^\beta}\,ds\\
&\le C(\omega)h^{H-\rho}
\int_{t_n}^{t_{n+1}}(s-t_n)^{-\beta}\,ds\\
&=\frac{C(\omega)}{1-\beta}
 h^{H-\rho-\beta+1}
\le C(\omega)h^{H-\rho-\beta+1}.
\end{align*}
Finally, if $t_n\le u<s<t_{n+1}$, then $C_h(u)=C_h(s)=t_n$, and hence
the discretized terms appearing in $G(s)-G(u)$ cancel. Thus,
\[
|G(s)-G(u)|
=|\sigma(s,X(s),Y(s))-\sigma(u,X(u),Y(u))|.
\]
By Hypothesis~1,
\begin{equation}\label{eq:I3sigma}
|G(s)-G(u)|
\le L_1|s-u|+L_2|X(s)-X(u)|+L_2|Y(s)-Y(u)|.
\end{equation}
Moreover, by Lemma~\ref{lem:3.1.1} and the $H-\rho$-Hölder regularity
of the extended path of $X$ on $[-r,T]$,
\[
|Y(s)-Y(u)|
\le C|s-u|+C(\omega)|s-u|^{H-\rho}.
\]
Substituting this estimate and \eqref{eq:I3sigma} into $I_3$, we obtain
\begin{align*}
I_3
&\le C(\omega)\int_{t_n}^{t_{n+1}}\int_{t_n}^{s}
\Big[(s-u)^{-\beta}
+(s-u)^{H-\rho-\beta-1}\Big]du\,ds\\
&\le C(\omega)h^{2-\beta}
+C(\omega)h^{H-\rho-\beta+1}.
\end{align*}
The second integral is finite because $H-\rho-\beta>0$. Also, since
$H-\rho<1$ and $0<h\le1$,
\[
h^{2-\beta}\le h^{H-\rho-\beta+1},
\]
and therefore
\[
I_3\le C(\omega)h^{H-\rho-\beta+1}.
\]
Combining the three estimates and using again $0<h\le1$,
\begin{align*}
\delta_{n+1}
&\le C(\omega)h^{H-\rho+1}
+C(\omega)h^{H-\rho-\beta+1}\\
&\le C(\omega)h^{H-\rho-\beta+1}.
\end{align*}
All the constants above are uniform for $t_n\in[0,T]$ and are therefore
independent of $n$. Taking the maximum over $n=0,\ldots,N-1$ gives
\eqref{Teo:delta}.
\end{proof}

\subsection{Convergence of the Backward-Euler Scheme}
\label{sec:convergencia-backward-4.6}

In Subsection~\ref{sec:consistencia-backward} we established the
consistency ( definition \ref{cnsistent}) of the backward-Euler scheme (Theorem~\ref{thm:4.1.2}): the
local error $\delta_{n+1}=|X(t_{n+1})-\widetilde X_h^{\,n+1}|$
satisfies, uniformly on the mesh,
\[
\max_{0\le n\le N-1}\delta_{n+1}
\le C(\omega)\,h^{H-\rho-\beta+1}, \qquad h\to0.
\]
The convergence analysis relies on a uniform a priori bound for
$\widetilde X_h$, on the fractional estimate for integrals with respect
to $B^H$ used in \cite{abreu2026euler}, and on a Gronwall inequality
with weakly singular kernel; see \cite{dixonmckee1986}. In what
follows we write
\[
\widetilde F(t):=\sup_{u\in[-r,t]}|\widetilde X_h(u)|.
\]

\begin{theorem}
\label{thm:convergencia-backward}
Suppose Hypotheses~$\phi$, 1, 2 and 3 hold, and let
\[
\beta\in(1-H,1/2),\qquad 0<\rho<H-\beta.
\]
There exists a set $\Omega_\rho\subset\Omega$ with
$\mathbb P(\Omega_\rho)=1$ such that, for every
$\omega\in\Omega_\rho$, there are measurable random variables
$h_\rho(\omega)>0$ and $C_\rho(\omega)<\infty$ for which, whenever
$0<h<h_\rho(\omega)$,
\begin{equation}\label{eq:global-convergence-backward}
\sup_{-r\le t\le T}\big|X(t)-\widetilde X_h(t)\big|
\le C_\rho(\omega)h^{H-\rho-\beta}.
\end{equation}
The random threshold may be chosen so that
$h_\rho(\omega)L_2<1$ and $h_\rho(\omega)L_3\le1/2$. In particular,
\[
\max_{0\le n\le N}\big|X(t_n)-\widetilde X_h(t_n)\big|
\le C_\rho(\omega)h^{H-\rho-\beta}.
\]
\end{theorem}

\begin{corollary}
\label{cor:highprob-backward}
Under the hypotheses of Theorem~\ref{thm:convergencia-backward}, for every
$\varepsilon\in(0,1)$ there exist deterministic constants
$h_{\varepsilon,\rho}>0$ and $C_{\varepsilon,\rho}<\infty$ and an event
$\Omega_{\varepsilon,\rho}$ with
$\mathbb P(\Omega_{\varepsilon,\rho})>1-\varepsilon$ such that, for all
$\omega\in\Omega_{\varepsilon,\rho}$ and $0<h<h_{\varepsilon,\rho}$,
\[
\sup_{-r\le t\le T}|X(t)-\widetilde X_h(t)|
\le C_{\varepsilon,\rho}h^{H-\rho-\beta}.
\]
\end{corollary}
\begin{proof}
On the full-probability set of Theorem~\ref{thm:convergencia-backward},
$h_\rho(\omega)>0$ and $C_\rho(\omega)<\infty$. Hence
\[
A_m:=\{h_\rho\ge m^{-1},\;C_\rho\le m\},\qquad m\ge1,
\]
increases to $\Omega_\rho$. Choose $m$ so that
$\mathbb P(A_m)>1-\varepsilon$ and take
$h_{\varepsilon,\rho}=m^{-1}$,
$C_{\varepsilon,\rho}=m$, and
$\Omega_{\varepsilon,\rho}=A_m$.
\end{proof}

The proof of the theorem is organized in three stages. First, we
establish the well-posedness of the implicit step and an exact
integral representation of the scheme at the mesh points. Next, we
obtain the a priori bounds and the global error estimates needed to
simultaneously control the pointwise error, its increments, and the
memory term. Finally, these estimates are combined into an integral
inequality for $\mathcal E_h$, which is closed by means of a Gronwall
inequality with weakly singular kernel.

The evaluation of the drift at $t_{n+1}$ makes relation
\eqref{Backward} implicit. The following proposition guarantees that
every step of the scheme is well defined.

\begin{proposition}
\label{prop:existencia-implicito}
Suppose $hL_2<1$. For each $n=0,\dots,N-1$, if the discrete history
$\{\widetilde X_h(t_j): -N\le j\le n\}$ is determined, then there
exists a unique value $\widetilde X_h(t_{n+1})$ satisfying
\eqref{Backward}. In particular, the discrete initial condition
determines, by induction, a unique sequence
$\{\widetilde X_h(t_n)\}_{n=-N}^{N}$.
\end{proposition}
\begin{proof}
By \eqref{Ytilde},
\[
\widetilde Y(t_{n+1})
= h\sum_{i=-N}^{-1}
K\bigl(t_{n+1},t_i,\widetilde X_h(t_{n+1+i})\bigr).
\]
Since $i\le -1$, we have $n+1+i\le n$; hence $\widetilde Y(t_{n+1})$
depends only on the already known discrete history. Likewise,
$\widetilde Y(t_n)$ is also determined.

For $x\in\mathbb R$, define
\[
\Psi_n(x):=\widetilde X_h(t_n)
+h\,b\bigl(t_{n+1},x,\widetilde Y(t_{n+1})\bigr) +\sigma\bigl(t_n,\widetilde X_h(t_n),\widetilde Y(t_n)\bigr)
\Delta B^H(t_n)
.
\]
Equation \eqref{Backward} is equivalent to finding a fixed point of
$\Psi_n$. By Hypothesis~1, for any $x,y\in\mathbb R$,
\[
|\Psi_n(x)-\Psi_n(y)|
\le hL_2|x-y|.
\]
Since $hL_2<1$, $\Psi_n$ is a contraction from $\mathbb R$ to
$\mathbb R$. The Banach fixed-point theorem provides a unique fixed
point, which is precisely $\widetilde X_h(t_{n+1})$. The last
assertion follows by induction from the discrete initial condition.
\end{proof}

The operators $C_h$ and $C_h^+$, defined in \eqref{Chu}-\eqref{Chu+},
allow us to express compactly the temporal asymmetry of the scheme: if
$u\in[t_n,t_{n+1})$, then $C_h(u)=t_n$ is the point where the diffusion
is evaluated and $C_h^+(u)=t_{n+1}$ is the point where the implicit
drift is evaluated.

Summing \eqref{Backward} for $n=0,\dots,m-1$ we obtain, for each mesh
point $t_m$, the exact identity
\begin{equation}\label{Backward-1}
\begin{aligned}
\widetilde X_h(t_m)
={}&X(0)
+ \int_0^{t_m}
 b\bigl(C_h^+(u),\widetilde X_h(C_h^+(u)),\widetilde Y(C_h^+(u))\bigr)\,du\\
&+ \int_0^{t_m}
 \sigma\bigl(C_h(u),\widetilde X_h(C_h(u)),\widetilde Y(C_h(u))\bigr)\,dB^H(u).
\end{aligned}
\end{equation}
Indeed, on each interval $[t_n,t_{n+1})$ the arguments of the
coefficients are constant. Consequently,
\[
\int_{t_n}^{t_{n+1}}
 b\bigl(C_h^+(u),\widetilde X_h(C_h^+(u)),\widetilde Y(C_h^+(u))\bigr)
\,du
= h\,b\bigl(t_{n+1},\widetilde X_h(t_{n+1}),
\widetilde Y(t_{n+1})\bigr),
\]
while
\[
\int_{t_n}^{t_{n+1}}
 \sigma\bigl(C_h(u),\widetilde X_h(C_h(u)),\widetilde Y(C_h(u))\bigr)
\,dB^H(u)
=\sigma\bigl(t_n,\widetilde X_h(t_n),\widetilde Y(t_n)\bigr)
\Delta B^H(t_n).
\]
Summing these identities and using $\widetilde X_h(0)=\phi(0)=X(0)$
recovers exactly the telescoping sum of \eqref{Backward}. In
particular, \eqref{Backward-1} does not introduce any new approximation
and does not require the use of the continuous interpolation between
mesh points.

By the definition of the auxiliary interpolation, the same
representation with upper limit $t$ holds for every $t\in[0,T]$. This
identity is understood pathwise and will be used only to obtain the a
priori estimates of the interpolation.

\begin{remark}[Comparison with the explicit Euler scheme]
\label{rem:explicit-backward}
Representation \eqref{Backward-1} allows us to identify precisely the
difference between the explicit Euler scheme studied in
\cite{abreu2026euler} and the backward-Euler scheme considered here. In
the explicit scheme, both the drift and the diffusion are evaluated at
the left endpoint of each discretization interval, through $C_h$. In
the backward-Euler scheme only the evaluation of the drift changes: it
is evaluated at $C_h^+(u)$, while the diffusion coefficient continues
to be evaluated at $C_h(u)$.

Consequently, the arguments of \cite{abreu2026euler} that depend only
on the diffusion, or on properties of the memory term that do not
involve the implicit evaluation of the drift, apply directly. On the
other hand, the estimates involving the drift must be revisited to
account for the shift from $C_h(u)$ to $C_h^+(u)$. This distinction
will be used systematically in the following results to separate the
arguments inherited from the explicit Euler scheme from those requiring
a specific adaptation to the backward-Euler scheme.
\end{remark}

The following estimates depend only on the structure of the delay term
and on Hypothesis~3. They are included to fix the notation used in the
subsequent estimates.

\begin{lemma}
\label{lem:3.1.1}
Let $Y$ be given by \eqref{eq:Y} and suppose Hypothesis~3 holds. Then,
for $0\le s\le t\le T$,
\[
|Y(t)-Y(s)|
\le rK_1|t-s|
+K_2\int_{-r}^{0}|X(t+u)-X(s+u)|\,du.
\]
\end{lemma}
\begin{proof}
Using the representation of $Y$ over the fixed interval $[-r,0]$,
\begin{align*}
Y(t)-Y(s)
={}&\int_{-r}^{0}
\big[K(t,u,X(t+u))-K(s,u,X(t+u))\big]\,du\\
&+\int_{-r}^{0}
\big[K(s,u,X(t+u))-K(s,u,X(s+u))\big]\,du.
\end{align*}
By Hypothesis~3,
\[
|K(t,u,X(t+u))-K(s,u,X(t+u))|\le K_1|t-s|,
\]
while
\[
|K(s,u,X(t+u))-K(s,u,X(s+u))|
\le K_2|X(t+u)-X(s+u)|.
\]
Integrating over $[-r,0]$ gives the estimate.
\end{proof}

\begin{lemma}
\label{lem:3.3.1}
Let $\widetilde Y$ be the Riemann-sum approximation defined in
\eqref{Ytilde}. Then, for every $u\in[0,T]$,
\begin{align}
|\widetilde Y(C_h(u))|
&\le rK_3\bigl(1+\widetilde F(u-h)\bigr),
\label{eq:Ytilde-left-bound}\\
|\widetilde Y(C_h^+(u))|
&\le rK_3\bigl(1+\widetilde F(C_h(u))\bigr)
\le rK_3\bigl(1+\widetilde F(u)\bigr).
\label{eq:Ytilde-right-bound}
\end{align}
\end{lemma}
\begin{proof}
The first inequality is Lemma~2 of \cite{abreu2026euler}, since the
definition of the discrete delay functional and its evaluation at
$C_h$ do not change. For the second, if $u\in[t_n,t_{n+1})$, then
$C_h^+(u)=t_{n+1}$ and, since $i\le-1$,
$t_{n+1}+t_i\le t_n=C_h(u)$. Hence,
\[
|\widetilde Y(C_h^+(u))|
\le \sum_{i=-N}^{-1}hK_3
\bigl(1+|\widetilde X_h(C_h^+(u)+t_i)|\bigr)
\le rK_3\bigl(1+\widetilde F(C_h(u))\bigr).
\]
For $u=T$ the same bound is obtained because the largest argument
appearing in $\widetilde Y(T)$ is $T-h$. The last inequality follows
from the monotonicity of $\widetilde F$.
\end{proof}

\begin{lemma}
\label{lem:3.3.2}
With $\widetilde Y$ given by \eqref{Ytilde},
\[
|\widetilde Y(u)-\widetilde Y(v)| \le rK_1|u-v| + K_2\int_{-r}^0
\big|\widetilde X_h(u+C_h(s))-\widetilde X_h(v+C_h(s))\big|\,ds,
\qquad 0\le v\le u\le T.
\]
\end{lemma}
\begin{proof}
This estimate is Lemma~3 of \cite{abreu2026euler}. Its proof uses only
Hypothesis~3 and the definition of $\widetilde Y$, which coincides
with the one used in the explicit Euler scheme.
\end{proof}

\begin{lemma}
\label{lem:3.3.3-back}
Let $v\in[t_n,t_{n+1})$ and suppose $hL_3\le 1/2$. Then there exists a
random constant $C(\omega)$, independent of $h$ and $n$, such that
\begin{align}
|\widetilde X_h(C_h^+(v))|
&\le C(\omega)\bigl(1+\widetilde F(C_h(v))\bigr),
\label{eq:forward-node-control}\\
|\widetilde X_h(v)-\widetilde X_h(C_h(v))|
&\le C(\omega)\bigl(1+\widetilde F(C_h(v))\bigr)
(v-C_h(v))^{H-\rho}.
\label{eq:one-step-backward}
\end{align}
\end{lemma}
\begin{proof}
Write $t_n=C_h(v)$. By Lemma~\ref{lem:3.3.1},
\[
|\widetilde Y(t_{n+1})|
\le rK_3\bigl(1+\widetilde F(t_n)\bigr),
\qquad
|\widetilde Y(t_n)|
\le rK_3\bigl(1+\widetilde F(t_n)\bigr).
\]
From recursion \eqref{Backward}, the linear growth of $b$ and
$\sigma$, and the $(H-\rho)$-Hölder regularity of $B^H$,
\begin{align*}
|\widetilde X_h(t_{n+1})|
\le{}& |\widetilde X_h(t_n)|
+hL_3\bigl(1+|\widetilde X_h(t_{n+1})|
+|\widetilde Y(t_{n+1})|\bigr)\\
&+L_3\bigl(1+|\widetilde X_h(t_n)|
+|\widetilde Y(t_n)|\bigr)
|B^H(t_{n+1})-B^H(t_n)|.
\end{align*}
Hence,
\[
(1-hL_3)|\widetilde X_h(t_{n+1})|
\le C(\omega)\bigl(1+\widetilde F(t_n)\bigr),
\]
and the condition $hL_3\le1/2$ allows the implicit term to be absorbed,
giving \eqref{eq:forward-node-control}.

Using now the auxiliary interpolation on $[t_n,t_{n+1})$,
\begin{align*}
|\widetilde X_h(v)-\widetilde X_h(t_n)|
\le{}& L_3\bigl(1+|\widetilde X_h(t_{n+1})|
+|\widetilde Y(t_{n+1})|\bigr)(v-t_n)\\
&+L_3\bigl(1+|\widetilde X_h(t_n)|
+|\widetilde Y(t_n)|\bigr)
|B^H(v)-B^H(t_n)|.
\end{align*}
Applying \eqref{eq:forward-node-control}, Lemma~\ref{lem:3.3.1}, and
the Hölder regularity of $B^H$, and using
$v-t_n\le T^{1-H+\rho}(v-t_n)^{H-\rho}$, we obtain
\eqref{eq:one-step-backward}. The estimate shows that the evaluation at
the right endpoint of the cell is controlled by $\widetilde F$ at the
left endpoint.
\end{proof}

The following estimate controls the fractional seminorm of the
increments of $\widetilde X_h$. The terms associated with the drift are
evaluated at $C_h^+$, while the diffusion terms retain the evaluation
at $C_h$.

\begin{lemma}
\label{lem:3.3.4-back}
\begin{align*}
&\int_0^{C_h(u)} |\widetilde X_h(C_h(u))-\widetilde X_h(v)|(u-v)^{-\beta-1}dv \\
&\le{} C(\omega)\Big(1+\int_0^{C_h(u)}|\widetilde X_h(w)|(u-w)^{-2\beta}dw
+\int_0^{C_h(u)}\widetilde F(w)(u-w)^{-2\beta}dw\\
&+\int_0^{C_h(u)}(u-w)^{-\beta}\int_0^{C_h(w)}\frac{|\widetilde X_h(C_h(w))-\widetilde
X_h(z)|}{(w-z)^{\beta+1}}dz\,dw + \widetilde F(C_h(u))\,h^{H-\rho-\beta}\\
&+\int_0^{C_h(u)}(u-w)^{-\beta}\int_0^{C_h(w)}\int_{-r}^0
\frac{|\widetilde X_h(C_h(w)+C_h(s))-\widetilde X_h(C_h(z)+C_h(s))|\,ds}{(w-z)^{\beta+1}}dz\,dw\Big).
\end{align*}
\end{lemma}
\begin{proof}
From \eqref{Backward-1}, Hypothesis~1, and the fractional estimate for
the integral with respect to $B^H$, for $0\le v\le C_h(u)$ we obtain
\begin{align*}
&|\widetilde X_h(C_h(u))-\widetilde X_h(v)|
\le C(\omega)\Bigg[
\int_v^{C_h(u)}
\bigl(1+|\widetilde X_h(C_h^+(w))|
      +|\widetilde Y(C_h^+(w))|\bigr)\,dw\\
&\quad+\int_v^{C_h(u)}
\frac{1+|\widetilde X_h(C_h(w))|+|\widetilde Y(C_h(w))|}
     {(w-v)^\beta}\,dw\\
&\quad+\int_v^{C_h(u)}\int_v^w
\frac{|C_h(w)-C_h(z)|
+|\widetilde X_h(C_h(w))-\widetilde X_h(C_h(z))|
+|\widetilde Y(C_h(w))-\widetilde Y(C_h(z))|}
{(w-z)^{\beta+1}}\,dz\,dw
\Bigg].
\end{align*}
The first term corresponds to the drift, while the remaining two arise
from the diffusion. Using $1\le T^\beta(w-v)^{-\beta}$, multiplying by
$(u-v)^{-\beta-1}$ and integrating with respect to $v$, we obtain the
same basic decomposition as in Lemma~5 of \cite{abreu2026euler}:
\[
\int_0^{C_h(u)}
\frac{|\widetilde X_h(C_h(u))-\widetilde X_h(v)|}
{(u-v)^{\beta+1}}\,dv
\le C(\omega)\sum_{j=1}^7Q_j(u),
\]
where
\begin{align*}
Q_1(u)&=\int_0^{C_h(u)}
(C_h(u)-v)^{1-\beta}(u-v)^{-\beta-1}\,dv,\\
Q_2(u)&=\int_0^{C_h(u)}(u-v)^{-\beta-1}
\int_v^{C_h(u)}
\frac{|\widetilde X_h(C_h^+(w))|}{(w-v)^\beta}\,dw\,dv,\\
Q_3(u)&=\int_0^{C_h(u)}(u-v)^{-\beta-1}
\int_v^{C_h(u)}
\frac{|\widetilde Y(C_h^+(w))|}{(w-v)^\beta}\,dw\,dv,\\
Q_4(u)&=h\int_0^{C_h(u)}(u-v)^{-\beta-1}
\int_v^{C_h(u)}(w-C_h(w))^{-\beta}\,dw\,dv,\\
Q_5(u)&=\int_0^{C_h(u)}(u-v)^{-\beta-1}
\int_v^{C_h(u)}\int_v^w
\frac{|\widetilde X_h(C_h(w))-\widetilde X_h(z)|}
{(w-z)^{\beta+1}}\,dz\,dw\,dv,\\
Q_6(u)&=\int_0^{C_h(u)}(u-v)^{-\beta-1}
\int_v^{C_h(u)}\int_v^w
\frac{|\widetilde X_h(z)-\widetilde X_h(C_h(z))|}
{(w-z)^{\beta+1}}\,dz\,dw\,dv,\\
Q_7(u)&=\int_0^{C_h(u)}(u-v)^{-\beta-1}
\int_v^{C_h(u)}\int_v^w
\frac{|\widetilde Y(C_h(w))-\widetilde Y(C_h(z))|}
{(w-z)^{\beta+1}}\,dz\,dw\,dv.
\end{align*}
The terms $Q_1,Q_4,Q_5$ and $Q_7$ do not contain the implicit
evaluation of the drift. Hence, they are estimated as in Lemma~5 of
\cite{abreu2026euler}. In particular,
\[
Q_1(u)\le C,\qquad Q_4(u)\le Ch^{1-\beta},
\]
\[
Q_5(u)\le\frac1\beta\int_0^{C_h(u)}(u-w)^{-\beta}
\int_0^{C_h(w)}
\frac{|\widetilde X_h(C_h(w))-\widetilde X_h(z)|}
{(w-z)^{\beta+1}}\,dz\,dw,
\]
and, by Lemma~\ref{lem:3.3.2},
\begin{align*}
Q_7(u)\le{}& C+\frac{K_2}{\beta}\int_0^{C_h(u)}(u-w)^{-\beta}\\
&\quad\times\int_0^{C_h(w)}\int_{-r}^0
\frac{|\widetilde X_h(C_h(w)+C_h(s))
-\widetilde X_h(C_h(z)+C_h(s))|}
{(w-z)^{\beta+1}}\,ds\,dz\,dw.
\end{align*}
The terms $Q_2$ and $Q_3$ contain the specific modification of the
backward-Euler scheme. By Fubini's theorem,
\[
Q_2(u)\le C\int_0^{C_h(u)}
\frac{|\widetilde X_h(C_h^+(w))|}{(u-w)^{2\beta}}\,dw.
\]
Estimate \eqref{eq:forward-node-control} implies
\[
Q_2(u)\le C(\omega)\left[
 u^{1-2\beta}
 +\int_0^{C_h(u)}
 \frac{\widetilde F(C_h(w))}{(u-w)^{2\beta}}\,dw
\right].
\]
Similarly, by \eqref{eq:Ytilde-right-bound},
\[
Q_3(u)\le C\left[
 u^{1-2\beta}
 +\int_0^{C_h(u)}
 \frac{\widetilde F(C_h(w))}{(u-w)^{2\beta}}\,dw
\right].
\]
Finally, $Q_6$ has the same algebraic form as in Lemma~5 of
\cite{abreu2026euler}, but the control of the intra-cell increment must
be carried out with Lemma~\ref{lem:3.3.3-back}. From
\eqref{eq:one-step-backward} and $z-C_h(z)\le h$ we obtain
\[
Q_6(u)\le
C(\omega)h^{H-\rho-\beta}
\bigl(1+\widetilde F(C_h(u))\bigr).
\]
Combining the seven estimates gives the stated inequality.
\end{proof}

\begin{lemma}
\label{lem:3.3.5-back}
For $w\in[0,T]$, with
$I(w):=\int_0^{C_h(w)}(w-z)^{-\beta-1}\int_{-r}^0|\widetilde
X_h(C_h(w)+C_h(s))-\widetilde X_h(C_h(z)+C_h(s))|\,ds\,dz$,
\begin{align*}
I(w) &\le C_1(\omega)\Big(1+\widetilde F(C_h(w))+h^{H-\rho}(w-C_h(w))^{-\beta}
+h(w-C_h(w))^{-\beta}\widetilde F(C_h(w)) \\
&+\int_0^{C_h(w)}(w-z)^{-\beta-1}\!\!\int_{-C_h(z)}^0\!\!|\widetilde
X_h(C_h(w)+C_h(s))-\widetilde X_h(C_h(z)+C_h(s))|\,ds\,dz\Big).
\end{align*}
\end{lemma}
\begin{proof}
The estimate uses the decomposition of the delay interval introduced
in Lemma~6 of \cite{abreu2026euler}. We recall it because this
partition allows us to distinguish the different regimes of the
shifted arguments. For $0\le z\le C_h(w)$,
\[
-r\le -C_h(w)\le -z\le -C_h(z)\le0,
\]
and hence
\[
[-r,0]=[-r,-C_h(w)]\cup[-C_h(w),-z]
\cup[-z,-C_h(z)]\cup[-C_h(z),0].
\]
On $[-r,-C_h(w)]$, both arguments $C_h(w)+C_h(s)$ and
$C_h(z)+C_h(s)$ belong to the initial interval. The $(H-\rho)$-Hölder
regularity of $\phi$ gives
\[
|\widetilde X_h(C_h(w)+C_h(s))
-\widetilde X_h(C_h(z)+C_h(s))|
\le C|C_h(w)-C_h(z)|^{H-\rho},
\]
and, after integrating with respect to $z$, this yields a bounded
contribution and the term
\[
C(\omega)h^{H-\rho}(w-C_h(w))^{-\beta}.
\]
On $[-C_h(w),-z]$, the shifted arguments may lie on either side of the
initial time. We introduce $\widetilde X_h(0)=\phi(0)$ to separate the
contributions associated with the initial condition and with the
numerical path. Both are controlled by the regularity of $\phi$ and by
$\widetilde F(C_h(w))$.

The interval $[-z,-C_h(z)]$ corresponds to the transition generated by
the mesh projection and has length
\[
z-C_h(z)\le h.
\]
The same size estimate then produces the additional term
\[
h(w-C_h(w))^{-\beta}\widetilde F(C_h(w)).
\]
Finally, on $[-C_h(z),0]$ both arguments belong to the numerical part
of the path. This contribution is not estimated crudely, but is kept
as
\[
\int_0^{C_h(w)}(w-z)^{-\beta-1}
\int_{-C_h(z)}^0
|\widetilde X_h(C_h(w)+C_h(s))
-\widetilde X_h(C_h(z)+C_h(s))|\,ds\,dz.
\]
The estimates of the first three regions are those of Lemma~6 of
\cite{abreu2026euler}; the fourth region is kept for the subsequent
recursive fractional control. Summing the four contributions gives the
inequality of the lemma.
\end{proof}

The following estimate combines the previous controls for the delayed
increments of the scheme.

\begin{lemma}
\label{lem:3.3.6-back}
\begin{align*}
&\int_0^{C_h(w)}\int_{-C_h(z)}^0\frac{|\widetilde
X_h(C_h(w)+C_h(s))-\widetilde X_h(C_h(z)+C_h(s))|}{(w-z)^{\beta+1}}\,ds\,dz \\
&\le C(\omega)\Big(1+h^{1-\beta}+h^{H-\rho-\beta}+h^{H-\rho-\beta}\widetilde
F(w) +h^{1-\beta}\widetilde F(C_h(w))
\end{align*}
\begin{align*}
&+\int_0^{C_h(w)}(w-u)^{-\beta}\widetilde F(C_h(u))\,du \\
&+\int_0^{C_h(w)}\frac{\widetilde F(C_h(z))}{(w-z)^{2\beta}}\,dz \\
&+\int_0^{C_h(w)}(w-u)^{-\beta}
\int_0^{C_h(u)}
\frac{|\widetilde X_h(C_h(u))-\widetilde X_h(v)|}{(u-v)^{\beta+1}}\,dv\,du \\
&+\int_0^{C_h(w)}(w-u)^{-\beta}
\int_0^{C_h(u)}(u-v)^{-\beta-1} \\
&\qquad\times\int_{-C_h(v)}^0
|\widetilde X_h(C_h(u)+C_h(s))-\widetilde X_h(C_h(v)+C_h(s))|\,ds\,dv\,du\Big).
\end{align*}
\end{lemma}
\begin{proof}
We write $|\widetilde X_h(C_h(w)+C_h(s))-\widetilde
X_h(C_h(z)+C_h(s))|\le|\widetilde X_h(C_h(w)+C_h(s))-\widetilde
X_h(z+C_h(s))|+|\widetilde X_h(z+C_h(s))-\widetilde
X_h(C_h(z)+C_h(s))|$. Since $C_h(z)+C_h(s)=C_h(z+C_h(s))$, the second
term is bounded using Lemma~\ref{lem:3.3.3-back}. Indeed,
$C_h(z+C_h(s))=C_h(z)+C_h(s)$ and $C_h(z+C_h(s))\le C_h(z)$, so
\[
|\widetilde X_h(z+C_h(s))-\widetilde X_h(C_h(z)+C_h(s))|
\le C(\omega)\big(1+\widetilde F(C_h(z))\big)h^{H-\rho}.
\]
This bound is combined with \eqref{eq:forward-node-control} to control
the values evaluated at the right endpoint of each cell. For the first
term we apply the fractional estimate for the integral with respect to
$B^H$ on the interval $[z+C_h(s),C_h(w)+C_h(s)]$:
\begin{align*}
&|\widetilde X_h(C_h(w)+C_h(s))-\widetilde X_h(z+C_h(s))|\\
&\le \Big|\int_{z+C_h(s)}^{C_h(w)+C_h(s)}
 b\big(C_h^+(v),\widetilde X_h(C_h^+(v)),\widetilde Y(C_h^+(v))\big)\,dv\Big|\\
&\quad +C(\omega)\int_{z+C_h(s)}^{C_h(w)+C_h(s)}
\frac{1+|\widetilde X_h(C_h(v))|+|\widetilde Y(C_h(v))|}
{(v-(z+C_h(s)))^\beta}\,dv\\
&\quad +C(\omega)\int_{z+C_h(s)}^{C_h(w)+C_h(s)}
\int_{z+C_h(s)}^{u}
\frac{|C_h(u)-C_h(v)|}{(u-v)^{\beta+1}}\,dv\,du\\
&\quad +C(\omega)\int_{z+C_h(s)}^{C_h(w)+C_h(s)}
\int_{z+C_h(s)}^{u}
\frac{|\widetilde X_h(C_h(u))-\widetilde X_h(C_h(v))|}{(u-v)^{\beta+1}}\,dv\,du\\
&\quad +C(\omega)\int_{z+C_h(s)}^{C_h(w)+C_h(s)}
\int_{z+C_h(s)}^{u}
\frac{|\widetilde Y(C_h(u))-\widetilde Y(C_h(v))|}{(u-v)^{\beta+1}}\,dv\,du.
\end{align*}
By \eqref{eq:forward-node-control} and \eqref{eq:Ytilde-right-bound},
\[
|b(C_h^+(v),\widetilde X_h(C_h^+(v)),\widetilde Y(C_h^+(v)))|
\le C(\omega)\bigl(1+\widetilde F(C_h(v))\bigr).
\]
Hence, the drift term satisfies an estimate in terms of
$1+\widetilde F(C_h(v))$. Using also
$1\le T^\beta(v-(z+C_h(s)))^{-\beta}$, the decomposition into the
contributions $R_1,\ldots,R_6$ is the same as in Lemma~7 of
\cite{abreu2026euler}. The terms $R_1,R_3$ and $R_4$ are estimated by
the same singular integrations. Contribution $R_2$ must be kept in the
form
\[
R_2(w)\le C\int_0^{C_h(w)}
\frac{\widetilde F(C_h(z))}{(w-z)^{2\beta}}\,dz.
\]
The backward adaptation appears in $R_5$, where the values at the
right endpoint are controlled by \eqref{eq:forward-node-control}, and
in the intra-cell increments, for which
Lemma~\ref{lem:3.3.3-back} is used. Finally, $R_6$ is estimated by
combining Lemma~\ref{lem:3.3.2} with the decomposition of
Lemma~\ref{lem:3.3.5-back}; in particular, the convolution
\[
\int_0^{C_h(w)}(w-u)^{-\beta}\widetilde F(C_h(u))\,du
\]
is retained. These modifications, together with the estimates
transferred from Lemma~7 of \cite{abreu2026euler}, give the stated
bound.
\end{proof}

\begin{theorem}
\label{thm:apriori-backward}
Suppose Hypotheses~$\phi$, 1 and 3 hold, and let
\[
\beta\in(1-H,1/2),\qquad 0<\rho<H-\beta.
\]
There exists a set $\Omega_\rho^{\rm ap}\subset\Omega$ with
$\mathbb P(\Omega_\rho^{\rm ap})=1$ such that, for every
$\omega\in\Omega_\rho^{\rm ap}$, there is
$h_\rho^{\rm ap}(\omega)>0$ and a finite constant $C_\rho(\omega)$,
independent of $h$, for which, whenever
$0<h<h_\rho^{\rm ap}(\omega)$,
\[
\sup_{t\in[-r,T]}|\widetilde X_h(t)|\le C_\rho(\omega).
\]
Moreover, the auxiliary interpolation satisfies
\[
|\widetilde X_h(t)-\widetilde X_h(s)|
\le C_\rho(\omega)|t-s|^{H-\rho},
\qquad -r\le s<t\le T.
\]
\end{theorem}
\begin{proof}
Recall that
\[
\widetilde F(t)=\sup_{u\in[-r,t]}|\widetilde X_h(u)|.
\]
For $t\in[0,T]$, define the functional
\begin{align}
\label{eq:phi-apriori-back}
\varphi_h(t):={}&\widetilde F(t)
+\int_0^{C_h(t)}
\frac{|\widetilde X_h(C_h(t))-\widetilde X_h(v)|}
{(t-v)^{\beta+1}}\,dv\nonumber\\
&+\int_0^{C_h(t)}\int_{-C_h(v)}^0
\frac{|\widetilde X_h(C_h(t)+C_h(s))
-\widetilde X_h(C_h(v)+C_h(s))|}
{(t-v)^{\beta+1}}\,ds\,dv.
\end{align}
Set
\[
\Omega_\rho^{\rm ap}
:=\left\{\omega:
[B^H(\omega)]_{H-\rho;[0,T]}<\infty\right\}.
\]
The Hölder regularity of fractional Brownian motion gives
$\mathbb P(\Omega_\rho^{\rm ap})=1$. Fix
$\omega\in\Omega_\rho^{\rm ap}$ for the remainder of the proof. All
constants produced by the fractional integral estimates are then finite
pathwise constants depending on $\omega$ and $\rho$, but not on $h$.

The integral representation of the scheme and the fractional estimate
for the integral with respect to $B^H$ give a bounded drift
contribution, via \eqref{eq:forward-node-control},
\eqref{eq:Ytilde-right-bound} and the linear growth of $b$, through
\[
|b(C_h^+(u),\widetilde X_h(C_h^+(u)),
\widetilde Y(C_h^+(u)))|
\le C(\omega)\bigl(1+\widetilde F(C_h(u))\bigr).
\]
Consequently,
\begin{equation}
\label{eq:deriva-apriori-back}
\int_0^t
|b(C_h^+(u),\widetilde X_h(C_h^+(u)),
\widetilde Y(C_h^+(u)))|\,du
\le C(\omega)\left(1+
\int_0^t u^{-\beta}\widetilde F(u)\,du\right).
\end{equation}
The diffusion continues to be evaluated at $C_h$. Applying the same
fractional decomposition as in \cite{abreu2026euler}, together with
Lemmas~\ref{lem:3.3.1}-\ref{lem:3.3.6-back}, and reordering the
integrals via Fubini's theorem, we obtain
\begin{equation}
\label{eq:F-apriori-back}
\widetilde F(t)
\le C_0(\omega)\left[
1+h^{H-\rho-\beta}\widetilde F(t)
+\int_0^t
\bigl(u^{-2\beta}+(t-u)^{-2\beta}\bigr)
\varphi_h(u)\,du
\right].
\end{equation}
This estimate uses $x^{-\beta}\le T^{\beta}x^{-2\beta}$ for
$0<x\le T$ and $2\beta<1$. In particular, the contribution associated
with term $R_2$ of Lemma~\ref{lem:3.3.6-back} satisfies
\[
\int_0^t\int_0^{C_h(u)}
\frac{\widetilde F(C_h(z))}{(u-z)^{2\beta}}\,dz\,du
\le C\int_0^t u^{-2\beta}\widetilde F(u)\,du,
\]
after interchanging the order of integration.

Lemma~\ref{lem:3.3.4-back} gives
\begin{align}
\label{eq:increment-apriori-back}
&\int_0^{C_h(t)}
\frac{|\widetilde X_h(C_h(t))-\widetilde X_h(v)|}
{(t-v)^{\beta+1}}\,dv\nonumber\\
&\hspace{1cm}\le C_0(\omega)\left[
1+h^{H-\rho-\beta}\widetilde F(t)
+\int_0^t (t-u)^{-2\beta}\varphi_h(u)\,du
\right].
\end{align}
On the other hand, Lemma~\ref{lem:3.3.6-back}, keeping the convolution
in $\widetilde F(C_h(u))$ and using $h^{1-\beta}\le h^{H-\rho-\beta}$
for $0<h\le1$, gives
\begin{align}
\label{eq:delay-apriori-back}
&\int_0^{C_h(t)}\int_{-C_h(v)}^0
\frac{|\widetilde X_h(C_h(t)+C_h(s))
-\widetilde X_h(C_h(v)+C_h(s))|}
{(t-v)^{\beta+1}}\,ds\,dv\nonumber\\
&\hspace{1cm}\le C_0(\omega)\left[
1+h^{H-\rho-\beta}\widetilde F(t)
+\int_0^t (t-u)^{-2\beta}\varphi_h(u)\,du
\right].
\end{align}
Adding \eqref{eq:F-apriori-back}, \eqref{eq:increment-apriori-back} and
\eqref{eq:delay-apriori-back}, and using
$\widetilde F(t)\le\varphi_h(t)$,
\begin{equation}
\label{eq:phi-gronwall-back}
\varphi_h(t)
\le C_1(\omega)\left[
1+h^{H-\rho-\beta}\varphi_h(t)
+\int_0^t
\bigl(u^{-2\beta}+(t-u)^{-2\beta}\bigr)
\varphi_h(u)\,du
\right].
\end{equation}
Since $H-\rho-\beta>0$, for the fixed
$\omega\in\Omega_\rho^{\rm ap}$ we may choose
$h_\rho^{\rm ap}(\omega)>0$ so that
\[
h_\rho^{\rm ap}(\omega)L_2<1,\qquad
h_\rho^{\rm ap}(\omega)L_3\le\frac12,\qquad
C_1(\omega)\bigl(h_\rho^{\rm ap}(\omega)\bigr)^{H-\rho-\beta}\le\frac12.
\]
For $0<h<h_\rho^{\rm ap}(\omega)$, the second term
on the right-hand side of \eqref{eq:phi-gronwall-back} is absorbed
into the left-hand side, leaving
\[
\varphi_h(t)
\le C_2(\omega)\left[
1+\int_0^t
\bigl(u^{-2\beta}+(t-u)^{-2\beta}\bigr)
\varphi_h(u)\,du
\right].
\]
Since $2\beta<1$, the functions $u\mapsto u^{-2\beta}$ and
$u\mapsto (t-u)^{-2\beta}$ are integrable at their respective
endpoints, uniformly for $t\in[0,T]$. Therefore, the continuous-case
extension of the Gronwall inequality for weakly singular kernels
established in \cite{dixonmckee1986}, and also used in
\cite{abreu2026euler}, is applicable and gives
\[
\sup_{t\in[0,T]}\varphi_h(t)\le C(\omega),
\]
and, by \eqref{eq:phi-apriori-back},
\[
\sup_{t\in[-r,T]}|\widetilde X_h(t)|\le C(\omega).
\]
It remains to establish the Hölder regularity. For $0\le s<t\le T$,
the integral representation of the interpolation gives
\begin{align*}
\widetilde X_h(t)-\widetilde X_h(s)
={}&\int_s^t b(C_h^+(u),\widetilde X_h(C_h^+(u)),
\widetilde Y(C_h^+(u)))\,du\\
&+\int_s^t\sigma(C_h(u),\widetilde X_h(C_h(u)),
\widetilde Y(C_h(u)))\,dB^H(u).
\end{align*}
The first integral is bounded by $C(\omega)(t-s)$. For the second, the
fractional estimate, together with \eqref{eq:phi-apriori-back} and
$\varphi_h(u)\le C(\omega)$, gives
\[
\left|\int_s^t\sigma(C_h(u),\widetilde X_h(C_h(u)),
\widetilde Y(C_h(u)))\,dB^H(u)\right|
\le C(\omega)(t-s)^{H-\rho}.
\]
Since $H-\rho<1$,
\[
(t-s)\le T^{1-H+\rho}(t-s)^{H-\rho},
\]
and we conclude
\[
|\widetilde X_h(t)-\widetilde X_h(s)|
\le C(\omega)(t-s)^{H-\rho},
\qquad 0\le s<t\le T.
\]
To extend the estimate to all of $[-r,T]$, observe that
$\widetilde X_h=\phi$ on $[-r,0]$. By Hypothesis~$\phi$,
\[
|\widetilde X_h(t)-\widetilde X_h(s)|
\le L_\rho|t-s|^{H-\rho},
\qquad -r\le s<t\le0.
\]
Finally, if $-r\le s<0<t\le T$, then
\begin{align*}
|\widetilde X_h(t)-\widetilde X_h(s)|
&\le |\widetilde X_h(t)-\widetilde X_h(0)|
   +|\phi(0)-\phi(s)|\\
&\le C(\omega)t^{H-\rho}+L_\rho|s|^{H-\rho}\\
&\le C(\omega)|t-s|^{H-\rho}.
\end{align*}
This completes the estimate on $[-r,T]$.
\end{proof}

\begin{corollary}
\label{cor:3.3.1-back}
Under the hypotheses of Theorem~\ref{thm:apriori-backward}, for
$0\le v\le u\le T$,
\[
|Y(u)-Y(v)|+|\widetilde Y(u)-\widetilde Y(v)|
\le C(\omega)|u-v|^{H-\rho}.
\]
\end{corollary}
\begin{proof}
For $Y$, Lemma~\ref{lem:3.1.1} and the $H-\rho$-Hölder regularity of
$X$ on $[-r,T]$ give the estimate. For $\widetilde Y$, we apply
Lemma~\ref{lem:3.3.2} together with Theorem~\ref{thm:apriori-backward}.
\end{proof}

\subsubsection{Error Functionals}
\label{sec:error-functionals-backward}

For $t\in[-r,T]$, define the error of the auxiliary interpolation by
\begin{equation}
\label{error-backward}
E_h(t):=X(t)-\widetilde X_h(t).
\end{equation}
In particular, $E_h(t)=0$ for $t\in[-r,0]$. For $0\le v<u\le T$ we
introduce the increments of the state error and of the memory-term
error,
\begin{align}
\label{DeltaX-backward}
\Delta^X_h(v,u)
&:=|E_h(u)-E_h(v)|,\\
\label{DeltaY-backward}
\Delta^Y_h(v,u)
&:=\big|Y(u)-\widetilde Y(u)-Y(v)+\widetilde Y(v)\big|.
\end{align}
The continuous uniform error is denoted by
\begin{equation}
\label{Zh-cont-backward}
Z_h(t):=\sup_{0\le s\le t}|E_h(s)|,
\qquad t\in[0,T].
\end{equation}
Since $E_h(t)=0$ for $t\in[-r,0]$, the uniform error over the whole
interval $[-r,T]$ satisfies
\[
\sup_{-r\le t\le T}|E_h(t)|=Z_h(T).
\]
In particular,
\[
\max_{0\le n\le N}|E_h(t_n)|\le Z_h(T).
\]
According to the fractional structure of the integral with respect to
$B^H$, we define
\begin{equation}
\label{theta-backward}
\theta_h(u)
:=\int_0^u
\frac{\Delta^X_h(v,u)}{(u-v)^{\beta+1}}\,dv,
\qquad u\in[0,T],
\end{equation}
and, to control the increments associated with the memory term,
\begin{equation}
\label{kappa-backward}
\kappa_h(u)
:=\int_0^{C_h(u)}\int_{-C_h(v)}^0
\frac{\Delta^X_h(v+s,u+s)}{(u-v)^{\beta+1}}\,ds\,dv.
\end{equation}
The integrals in \eqref{theta-backward} and \eqref{kappa-backward} are
finite by virtue of the $H-\rho$-Hölder regularity of $X$ and of
$\widetilde X_h$, together with the condition $\beta<H-\rho$.

Finally, we introduce the error functional
\begin{equation}
\label{Eh-functional-backward}
\mathcal E_h(t):=Z_h(t)+\theta_h(t)+\kappa_h(t),
\qquad t\in[0,T].
\end{equation}
This functional is different from the functional $\varphi_h$ used in
Theorem~\ref{thm:apriori-backward}: $\varphi_h$ controls the size and
regularity of the approximation, while $\mathcal E_h$ is built
exclusively from differences between the exact solution and the
numerical approximation.

The proof of convergence is obtained by jointly controlling $Z_h$,
$\theta_h$ and $\kappa_h$. The part associated with the diffusion
retains the structure of the analysis for the explicit Euler scheme of
\cite{abreu2026euler}, while the drift estimates must incorporate the
evaluation at $C_h^+$. In particular, the goal is to obtain a closed
integral inequality for $\mathcal E_h$ whose free term is of order
$h^{H-\rho-\beta}$.

\subsubsection{Global Error Estimates}
\label{sec:error-estimates-backward}

In what follows, the positive constants denoted by $C$ are
deterministic, while $C(\omega)$ may change from line to line and is
independent of $h$.

\begin{lemma}
\label{lem:delay-error-backward}
Under the hypotheses of Theorem~\ref{thm:convergencia-backward}, for
every $t\in[0,T]$,
\begin{equation}
\label{eq:delay-point-error-backward}
|Y(t)-\widetilde Y(t)|
\le C Z_h(t)+C(\omega)h^{H-\rho}.
\end{equation}
Moreover,
\begin{equation}
\label{eq:delay-increment-error-backward}
\int_0^t\int_0^u
\frac{\Delta_h^Y(v,u)}{(u-v)^{\beta+1}}\,dv\,du
\le C(\omega)h^{H-\rho-\beta}
 +C\int_0^t Z_h(u)\,du
 +C\int_0^t \kappa_h(u)\,du .
\end{equation}
\end{lemma}
\begin{proof}
From the definitions of $Y$ and $\widetilde Y$,
\begin{align*}
|Y(t)-\widetilde Y(t)|
&\le \int_{-r}^0\Bigl(
K_1|s-C_h(s)|
 +K_2|X(t+s)-\widetilde X_h(t+C_h(s))|
\Bigr)\,ds\\
&\le Ch+C\int_{-r}^0|E_h(t+s)|\,ds
 +C\int_{-r}^0
 |\widetilde X_h(t+s)-\widetilde X_h(t+C_h(s))|\,ds.
\end{align*}
The common initial condition implies $E_h=0$ on $[-r,0]$, and the
$H-\rho$-Hölder regularity of $\widetilde X_h$, together with
$|s-C_h(s)|\le h$, gives \eqref{eq:delay-point-error-backward}.

We next prove the increment estimate. For $0\le v<u\le t$, apply
Lemma~\ref{lem:K-second-increment} to
\begin{align*}
&K(u,s,X(u+s))-K(u,C_h(s),\widetilde X_h(u+C_h(s)))\\
&\qquad{}-K(v,s,X(v+s))
 +K(v,C_h(s),\widetilde X_h(v+C_h(s))).
\end{align*}
After integration in $s\in[-r,0]$ this yields
\begin{equation}
\label{eq:delay-seven-terms}
\int_0^u
\frac{\Delta_h^Y(v,u)}{(u-v)^{\beta+1}}\,dv
\le C\sum_{j=1}^7 M_j(u),
\end{equation}
where the seven nonnegative quantities $M_j$ are displayed explicitly
in Appendix~\ref{app:delay-seven-terms}. They separate the time/memory
quadrature error, the pointwise state error, the genuine increment
error, and the product terms generated by the second-increment bound
for $K$.

The elementary terms satisfy
\begin{align}
\label{eq:delay-M-easy}
\int_0^t M_1(u)\,du+\int_0^t M_5(u)\,du
&\le C(\omega)h,\\
\int_0^t\{M_2(u)+M_3(u)+M_6(u)+M_7(u)\}\,du
&\le C\int_0^t Z_h(u)\,du+C(\omega)h^{H-\rho}.
\end{align}
The only term requiring a further decomposition is $M_4$. Splitting
its $(v,s)$-domain according to whether the shifted arguments lie in
the initial interval or in $[0,T]$ gives six regions. The region
estimates proved in Appendix~\ref{app:delay-seven-terms} yield
\begin{equation}
\label{eq:delay-M4-bound}
\int_0^t M_4(u)\,du
\le C(\omega)h^{H-\rho-\beta}
 +C\int_0^t Z_h(u)\,du
 +C\int_0^t\kappa_h(u)\,du.
\end{equation}
Combining \eqref{eq:delay-seven-terms}-\eqref{eq:delay-M4-bound}, and
using $0<h\le1$ and
$0<H-\rho-\beta<H-\rho<1$, so that both $h$ and
$h^{H-\rho}$ are bounded by $h^{H-\rho-\beta}$, gives
\eqref{eq:delay-increment-error-backward}.
\end{proof}

\begin{lemma}
\label{lem:diffusion-transfer-backward}
Set
\[
G_h(s):=\sigma(s,X(s),Y(s))
-\sigma(C_h(s),\widetilde X_h(C_h(s)),\widetilde Y(C_h(s)))
\]
and
\[
I_{\sigma,h}(t):=\int_0^tG_h(s)\,dB^H(s).
\]
Under the hypotheses of Theorem~\ref{thm:convergencia-backward}, the
following bounds hold for $t,u\in[0,T]$:
\begin{align}
\label{eq:diffusion-sup-transfer}
\sup_{0\le y\le t}|I_{\sigma,h}(y)|
\le C(\omega)\Bigg[{}&h^{H-\rho-\beta}
+\int_0^t Z_h(s)s^{-\beta}\,ds
+\int_0^t\theta_h(s)\,ds\notag\\
&+\int_0^t\int_0^s
\frac{\Delta_h^Y(v,s)}{(s-v)^{\beta+1}}\,dv\,ds
+\int_0^t Z_h(s)\,ds\Bigg],
\end{align}
\begin{align}
\label{eq:diffusion-theta-transfer}
\int_0^u\frac{|I_{\sigma,h}(u)-I_{\sigma,h}(v)|}
{(u-v)^{\beta+1}}\,dv
\le C(\omega)\Bigg[{}&h^{H-\rho-\beta}
+\int_0^u\mathcal E_h(s)(u-s)^{-\beta}\,ds\notag\\
&+\int_0^u Z_h(s)(u-s)^{-2\beta}\,ds\Bigg],
\end{align}
and the same upper bound controls the shifted-memory integral
\begin{equation}
\label{eq:diffusion-kappa-transfer}
\int_0^{C_h(u)}\int_{-C_h(v)}^0
\frac{|I_{\sigma,h}(u+s)-I_{\sigma,h}(v+s)|}
{(u-v)^{\beta+1}}\,ds\,dv.
\end{equation}
\end{lemma}
\begin{proof}
The proof is pathwise and uses only the left-point projection $C_h$;
the implicit projection $C_h^+$ never enters. Apply the generalized
Lebesgue-Stieltjes estimate to $G_h$. Hypothesis~1,
\eqref{eq:delay-point-error-backward}, and the Hölder regularity of
$X$ and $\widetilde X_h$ yield
\[
|G_h(s)|\le C\,Z_h(s)+C(\omega)h^{H-\rho}.
\]
For $0\le v<s\le T$, add and subtract the pointwise coefficients at
$(s,\widetilde X_h(C_h(s)),\widetilde Y(C_h(s)))$ and at the
corresponding time $v$. Lemma~\ref{lem:sigma-second-increment}, whose
assumptions follow from the functional Hypothesis~2 by
Lemma~\ref{lem:sigma-restriction}, controls the resulting four-point
difference. This gives contributions involving $\Delta_h^X$,
$\Delta_h^Y$, $Z_h$, and the projection remainder of size
$h^{H-\rho}$. Integration against the fractional kernels, followed by
Fubini's theorem, gives \eqref{eq:diffusion-sup-transfer}-
\eqref{eq:diffusion-kappa-transfer}. The singular integrations and the
same-cell/different-cell split for $C_h$ are written out in
Appendix~\ref{app:transfer-details}. These are the diffusion estimates
used in the explicit analysis of~\cite{abreu2026euler}, reproduced here
in the notation of the present paper so that the convergence argument
is self-contained.
\end{proof}

\begin{lemma}
\label{lem:Zh-backward}
Under the hypotheses of Theorem~\ref{thm:convergencia-backward}, for
every $t\in[0,T]$,
\begin{equation}
\label{eq:Zh-backward}
Z_h(t)
\le C(\omega)\left[
 h^{H-\rho-\beta}
 +\int_0^t (1+u^{-\beta})Z_h(u)\,du
 +\int_0^t\theta_h(u)\,du
 +\int_0^t\kappa_h(u)\,du
\right].
\end{equation}
\end{lemma}
\begin{proof}
Subtracting the integral representation of $\widetilde X_h$ \eqref{Backward-1} from the
exact equation \eqref{eq:X}, for $t\in[0,T]$ we obtain
\begin{align}
\label{eq:error-identity-backward}
E_h(t)
={}&\int_0^t\Bigl[
 b(s,X(s),Y(s))
 -b(C_h^+(s),\widetilde X_h(C_h^+(s)),
        \widetilde Y(C_h^+(s)))
\Bigr]ds\nonumber\\
&+\int_0^t\Bigl[
 \sigma(s,X(s),Y(s))
 -\sigma(C_h(s),\widetilde X_h(C_h(s)),
          \widetilde Y(C_h(s)))
\Bigr]dB^H(s).
\end{align}
Denote by $D_b(t)$ and $D_\sigma(t)$ the drift and diffusion
contributions, respectively.

For the drift, we introduce $(\widetilde X_h(s),\widetilde Y(s))$ as an
intermediate state. By Hypothesis~1,
\begin{align*}
|D_b(t)|
\le{}& C\int_0^t\bigl(|E_h(s)|+|Y(s)-\widetilde Y(s)|\bigr)\,ds
 +C\int_0^t|s-C_h^+(s)|\,ds\\
&+C\int_0^t\Bigl(
 |\widetilde X_h(s)-\widetilde X_h(C_h^+(s))|
 +|\widetilde Y(s)-\widetilde Y(C_h^+(s))|
\Bigr)\,ds.
\end{align*}
By \eqref{eq:delay-point-error-backward}, Theorem
\ref{thm:apriori-backward}, Corollary~\ref{cor:3.3.1-back} and
$|s-C_h^+(s)|\le h$,
\begin{equation}
\label{eq:deriva-error-backward}
|D_b(t)|
\le C\int_0^t Z_h(s)\,ds+C(\omega)h^{H-\rho}.
\end{equation}
For the diffusion contribution, Lemma~\ref{lem:diffusion-transfer-backward}
gives
\begin{align}
\label{eq:diffusion-error-backward}
\sup_{0\le y\le t}|D_\sigma(y)|
\le C(\omega)\Bigg[{}&h^{H-\rho-\beta}
 +\int_0^t Z_h(u)u^{-\beta}\,du
 +\int_0^t\theta_h(u)\,du\nonumber\\
&+\int_0^t\int_0^u
\frac{\Delta_h^Y(v,u)}{(u-v)^{\beta+1}}\,dv\,du
 +\int_0^t Z_h(u)\,du
\Bigg].
\end{align}
Taking the supremum in \eqref{eq:error-identity-backward}, combining
\eqref{eq:deriva-error-backward}-\eqref{eq:diffusion-error-backward}
and applying \eqref{eq:delay-increment-error-backward}, we obtain
\eqref{eq:Zh-backward}.
\end{proof}

\begin{lemma} 
\label{lem:theta-backward}
Under the hypotheses of Theorem~\ref{thm:convergencia-backward}, for
every $u\in[0,T]$,
\begin{align}
\label{eq:theta-est-backward}
\theta_h(u)
\le C(\omega)\Bigg[{}&h^{H-\rho-\beta}
 +\int_0^u
 \bigl(Z_h(s)+\theta_h(s)+\kappa_h(s)\bigr)
 (u-s)^{-\beta}\,ds\nonumber\\
&+\int_0^u Z_h(s)(u-s)^{-2\beta}\,ds
\Bigg].
\end{align}
\end{lemma}
\begin{proof}
Applying \eqref{eq:error-identity-backward} on $[v,u]$ and then
integrating with respect to $v$ with weight $(u-v)^{-\beta-1}$, the
drift contribution is estimated using the same decomposition as in
\eqref{eq:deriva-error-backward}. By Fubini's theorem,
\begin{align}
\label{eq:theta-deriva-backward}
&\int_0^u (u-v)^{-\beta-1}
 \int_v^u
 \left|b(s,X(s),Y(s))
 -b(C_h^+(s),\widetilde X_h(C_h^+(s)),
        \widetilde Y(C_h^+(s)))\right|ds\,dv\nonumber\\
&\hspace{1cm}\le
 C\int_0^u Z_h(s)(u-s)^{-\beta}\,ds
 +C(\omega)h^{H-\rho}.
\end{align}
For the diffusion contribution, apply
\eqref{eq:diffusion-theta-transfer} from
Lemma~\ref{lem:diffusion-transfer-backward} and then use
\eqref{eq:delay-increment-error-backward}. We obtain
\begin{align*}
C(\omega)\Bigg[
 h^{H-\rho-\beta}
 +\int_0^u
 (Z_h(s)+\theta_h(s)+\kappa_h(s))(u-s)^{-\beta}\,ds
 +\int_0^u Z_h(s)(u-s)^{-2\beta}\,ds
\Bigg].
\end{align*}
Since $h^{H-\rho}\le h^{H-\rho-\beta}$ for $0<h\le1$, we obtain
\eqref{eq:theta-est-backward}.
\end{proof}

\begin{lemma}
\label{lem:kappa-backward}
Under the hypotheses of Theorem~\ref{thm:convergencia-backward}, for
every $u\in[0,T]$,
\begin{align}
\label{eq:kappa-est-backward}
\kappa_h(u)
\le C(\omega)\Bigg[{}&h^{H-\rho-\beta}
 +\int_0^u
 \bigl(Z_h(s)+\theta_h(s)+\kappa_h(s)\bigr)
 (u-s)^{-\beta}\,ds\nonumber\\
&+\int_0^u Z_h(s)(u-s)^{-2\beta}\,ds
\Bigg].
\end{align}
\end{lemma}
\begin{proof}
For $0\le v\le C_h(u)$ and $-C_h(v)\le s\le0$ we have $v+s\ge0$. We
apply \eqref{eq:error-identity-backward} to the increment
\[
E_h(u+s)-E_h(v+s) = \big[D_b(u+s)-D_b(v+s)\big] + \big[D_\sigma(u+s)-D_\sigma(v+s)\big].
\] and then integrate with the weight
$(u-v)^{-\beta-1}$ appearing in \eqref{kappa-backward}.

Applying the triangle inequality, $\kappa_h(u) \le D_{b,\kappa}(u) + D_{\sigma,\kappa}(u)$, where
\begin{eqnarray}\label{Dbkappa}
D_{b,\kappa}(u) &:=& \int_0^{C_h(u)}\int_{-C_h(v)}^{0}
\frac{|D_b(u+s)-D_b(v+s)|}{(u-v)^{\beta+1}}dsdv,
\nonumber  \\ 
D_{\sigma,\kappa}(u) &:=& \int_0^{C_h(u)}\int_{-C_h(v)}^{0}
\frac{|D_\sigma(u+s)-D_\sigma(v+s)|}{(u-v)^{\beta+1}}dsdv.
\end{eqnarray}

The drift contribution is controlled, as in
\eqref{eq:theta-deriva-backward}, using Hypothesis~1,
\eqref{eq:delay-point-error-backward},
Theorem~\ref{thm:apriori-backward}, and
Corollary~\ref{cor:3.3.1-back}. A change in the order of integration gives
\begin{equation}
\label{eq:kappa-deriva-backward}
D_{b,\kappa}(u)
\le
C\int_0^u Z_h(s)(u-s)^{-\beta}\,ds
+
C(\omega)h^{H-\rho}.
\end{equation}

For the diffusion contribution, use the shifted estimate
\eqref{eq:diffusion-kappa-transfer} of
Lemma~\ref{lem:diffusion-transfer-backward}, followed by
\eqref{eq:delay-increment-error-backward}. We obtain
\begin{align*}
C(\omega)\Bigg[
 h^{H-\rho-\beta}
 +\int_0^u
 (Z_h(s)+\theta_h(s)+\kappa_h(s))(u-s)^{-\beta}\,ds
 +\int_0^u Z_h(s)(u-s)^{-2\beta}\,ds
\Bigg].
\end{align*}
Finally, $h^{H-\rho}\le h^{H-\rho-\beta}$ for $0<h\le1$, and we obtain
\eqref{eq:kappa-est-backward}.
\end{proof}
Lemmas 4.15-4.19 provide, respectively, the control of the delay functional, the
diffusion contribution, and the drift and memory contributions to the pointwise
error. We now combine these estimates into the closed integral inequality for
$\mathcal E_h$ announced at the beginning of this subsection, and close it by means
of the fractional Gronwall inequality of Dixon and McKee~\cite{dixonmckee1986}.
\begin{proof}[\textbf{Proof of Theorem~\ref{thm:convergencia-backward}}]
Let $\Omega_\rho$ be the intersection of the probability-one set on
which Theorem~\ref{thm:4.1.1} gives the $H-\rho$ Hölder regularity of
the exact solution and the set $\Omega_\rho^{\rm ap}$ of
Theorem~\ref{thm:apriori-backward}. Fix $\omega\in\Omega_\rho$.
All constants in the estimates below are then finite pathwise. We take
$h_\rho(\omega)$ smaller, if necessary, than the random threshold of
Theorem~\ref{thm:apriori-backward} and than the deterministic
thresholds required for the implicit step and the right-node estimate.
The constants and thresholds constructed in the proof can be chosen as
measurable functions of the Hölder seminorms entering the preceding
pathwise estimates.

Adding estimates \eqref{eq:Zh-backward}, \eqref{eq:theta-est-backward}
and \eqref{eq:kappa-est-backward}, and using definition
\eqref{Eh-functional-backward}, we obtain, for every $t\in[0,T]$,
\begin{align}
\label{eq:error-gronwall-backward}
\mathcal E_h(t)
\le C(\omega)\Bigg[{}&h^{H-\rho-\beta}
 +\int_0^t \bigl(1+s^{-\beta}\bigr)\mathcal E_h(s)\,ds\nonumber\\
&+\int_0^t \mathcal E_h(s)(t-s)^{-\beta}\,ds
 +\int_0^t \mathcal E_h(s)(t-s)^{-2\beta}\,ds
\Bigg].
\end{align}
Indeed, $Z_h\le\mathcal E_h$, and each integral appearing in
\eqref{eq:Zh-backward} is bounded by the corresponding integral of
$\mathcal E_h$. Likewise, the estimates for $\theta_h$ and $\kappa_h$
are incorporated directly into the last two terms of
\eqref{eq:error-gronwall-backward}.

Since $\beta<1/2$, the functions
$u\mapsto u^{-\beta}$ and $u\mapsto u^{-2\beta}$ belong to
$L^1(0,T)$. Thus the two convolution kernels in
\eqref{eq:error-gronwall-backward} are weakly singular and integrable,
and the additional coefficient $1+s^{-\beta}$ is integrable on
$(0,T)$. The nonnegative function $\mathcal E_h$ therefore satisfies
the hypotheses of the continuous weakly singular Gronwall inequality
of Dixon and McKee~\cite{dixonmckee1986}, after combining the regular
integral term with the two integrable Volterra kernels. Applying that
result to \eqref{eq:error-gronwall-backward} gives, for each fixed
$\omega$ in the full-probability set on which the preceding pathwise
estimates hold and for $0<h<h_\rho(\omega)$,
\begin{equation}
\label{eq:error-final-backward}
\sup_{0\le t\le T}\mathcal E_h(t)
\le C_\rho(\omega)h^{H-\rho-\beta},
\end{equation}
where $C_\rho(\omega)$ is finite and independent of $h$.

Finally, by \eqref{Zh-cont-backward},
\eqref{Eh-functional-backward}, and the fact that
$X(t)=\widetilde X_h(t)=\phi(t)$ for $t\in[-r,0]$,
\[
\sup_{-r\le t\le T}|X(t)-\widetilde X_h(t)|
=Z_h(T)
\le \mathcal E_h(T)
\le C_\rho(\omega)h^{H-\rho-\beta}.
\]
The nodal estimate follows immediately by restricting the supremum to
the mesh points. This proves the result.
\end{proof}

\begin{corollary}
\label{cor:2H-1}
For every $\eta\in(0,2H-1)$ one can choose
$\beta\in(1-H,1/2)$ and $\rho\in(0,H-\beta)$ so that
\[
H-\rho-\beta>2H-1-\eta.
\]
Consequently, on a set of probability one, for every $\omega$ there
exist $h_\eta(\omega)>0$ and $C_\eta(\omega)<\infty$ such that, for
$0<h<h_\eta(\omega)$,
\[
\sup_{-r\le t\le T}\big|X(t)-\widetilde X_h(t)\big|
\le C_\eta(\omega)h^{2H-1-\eta}.
\]
Thus, the auxiliary interpolation of the backward-Euler scheme
converges uniformly with any order strictly smaller than $2H-1$.
\end{corollary}
\begin{proof}
Given $\eta\in(0,2H-1)$, choose
$\beta\in(1-H,1/2)$ so that
\[
0<\beta-(1-H)<\frac{\eta}{2},
\]
and then $\rho>0$ such that
\[
0<\rho<\min\left\{H-\beta,\frac{\eta}{2}\right\}.
\]
Then
\[
H-\rho-\beta
=2H-1-\bigl(\beta-(1-H)\bigr)-\rho
>2H-1-\eta.
\]
The result follows from Theorem~\ref{thm:convergencia-backward},
since for $0<h<1$,
\[
h^{H-\rho-\beta}\le h^{2H-1-\eta}.
\]
\end{proof}

\begin{remark}
Corollary~\ref{cor:2H-1} does not establish a bound with exponent
exactly equal to $2H-1$. The strict restrictions
$\beta>1-H$ and $\rho>0$ allow any exponent below $2H-1$ to be
attained, while the limiting value is reached only as
$\beta\downarrow1-H$ and $\rho\downarrow0$.
\end{remark}

\section{Numerical Study}
\label{sec:5}

We consider an SFDE example satisfying the hypotheses of the analysis
and compare the backward-Euler and explicit Euler approximations on the
same realization of the noise. The purpose of the study is to
illustrate the numerical behaviour of both schemes; the empirical rates
are computed with respect to a reference solution on a fine mesh and
should not be interpreted as estimates of an exact asymptotic rate.

\subsection{Model Considered}

To illustrate scheme \eqref{Backward}, we consider the following
scalar SFDE with distributed delay
\begin{equation}\label{eq:example-sfde}
dX(t) = \big[-a\,X(t) + c\,Y(t)\big]\,dt + \sigma\big(X(t)\big)\,dB^H(t),
\qquad t\in[0,T], \qquad X(t)=x_0,\ t\in[-r,0],
\tag{E.1}
\end{equation}
with memory functional
\[
Y(t) = \kappa\int_{-r}^0 X(t+s)\,ds = \kappa\int_{t-r}^{t}X(u)\,du,
\]
that is, in the notation of \eqref{eq:X}-\eqref{eq:Y},
\[
b(t,x,y) = -a x + c y, \qquad \sigma(t,x,y) = \sigma_0+\sigma_1 x,
\qquad K(t,s,x) = \kappa x, \qquad \phi(t)\equiv x_0.
\]
These are the simplest choices that, while remaining fully explicit,
genuinely satisfy Hypotheses~$\phi$, 1, 2 and 3: $b$ and $\sigma$ are
affine, hence satisfying the Lipschitz and linear-growth conditions. In
the functional formulation of Hypothesis~2 one can identify
$\sigma(t,\xi)=\sigma_0+\sigma_1\xi_1(0)$, whose Fréchet derivative is
the linear functional $\eta\mapsto\sigma_1\eta_1(0)$. Moreover, for
$K(t,s,x)=\kappa x$ we have
$\nabla_{(s,x)}K(t,s,x)=(0,\kappa)$; hence this gradient is uniformly
bounded and its Lipschitz constant is zero. The kernel does not depend
on $(t,s)$ (so $K_1=0$), and $\phi$ is constant (Hölder with any
exponent). We use the numerical values
\[
H=0.7,\quad r=T=1,\quad a=1,\quad c=0.3,\quad \kappa=0.5,\quad
\sigma_0=0.25,\quad \sigma_1=0.15,\quad x_0=1.
\]
The diffusion is taken affine in $x$ and non-constant. If $\sigma$ were
constant, the stochastic integral would be reproduced exactly on each
mesh interval, $\int_{t_n}^{t_{n+1}}\sigma_0\,dB^H=\sigma_0\Delta
B^H(t_n)$, and this source of discretization error would disappear. The
choice $\sigma_1\neq0$ keeps a non-trivial stochastic contribution in
the numerical comparison.

Since $b$ is linear in $x$, the implicit step \eqref{Backward} has a
closed-form solution, with no need to iterate a fixed point:
\begin{equation}\label{eq:example-backward}
\widetilde X_h(t_{n+1}) = \frac{\widetilde X_h(t_n) + c\,h\,\widetilde
Y(t_{n+1}) + \sigma(\widetilde X_h(t_n))\,\Delta B^H(t_n)}{1+a\,h}.
\tag{E.2}
\end{equation}
Note that $\widetilde Y(t_{n+1})$, defined in \eqref{Ytilde}, only uses
values of $\widetilde X_h$ already computed
(Proposition~\ref{prop:existencia-implicito} of
Section~\ref{sec:convergencia-backward-4.6}), so
\eqref{eq:example-backward} is directly computable. In this example,
the contraction constant of the implicit map with respect to
$\widetilde X_h(t_{n+1})$ is $ha$; hence it suffices to impose
$ha<1$, a condition satisfied by all the step sizes used.

\subsection{Simulation Design}

\begin{itemize}
\item \textbf{Fractional Brownian motion.} A path of $B^H$ is
simulated on a fine mesh of $N_{\text{fine}}=2048$ points via the
Cholesky factorization of the covariance matrix
$R(s,t)=\tfrac12(|s|^{2H}+|t|^{2H}-|t-s|^{2H})$. For each coarser mesh
(with $N$ a divisor of $N_{\text{fine}}$), the same fine path is
\emph{subsampled} at the corresponding points, so that all schemes
being compared use exactly the same noise.
\item \textbf{Schemes.} We implement the explicit Euler scheme of
\cite{abreu2026euler} and the backward-Euler scheme
\eqref{eq:example-backward}, both with the discrete delay functional
$\widetilde Y$.
\item \textbf{Reference solution.} Since no closed-form exact solution
is available for \eqref{eq:example-sfde}, we use as a numerical
reference the backward-Euler approximation computed on the fine mesh
$N_{\text{fine}}=2048$. This choice should be kept in mind when
interpreting the comparison between the two schemes.
\item \textbf{Convergence study.} For $N\in\{8,16,32,64,128\}$ we
compute the maximum error over the mesh points,
$\max_n|X_{\text{ref}}(t_n)-\widetilde X_h(t_n)|$, averaged over
$M=24$ independent realizations of $B^H$ (to reduce the Monte Carlo
noise of a single path), and we estimate the empirical order of
convergence as the slope of the log-log regression of the mean error
against $h$.
\end{itemize}

\subsection{Numerical Results}

Table~\ref{tab:errores} and Figure~\ref{fig:resultados} show the
results.

\begin{table}[ht]
\centering
\begin{tabular}{@{}rrrr@{}}
\toprule
$N$ & $h$ & mean error backward-Euler & mean error explicit Euler \\
\midrule
8   & 0.12500 & $1.766\times10^{-2}$ & $2.817\times10^{-2}$ \\
16  & 0.06250 & $8.127\times10^{-3}$ & $1.473\times10^{-2}$ \\
32  & 0.03125 & $3.510\times10^{-3}$ & $8.145\times10^{-3}$ \\
64  & 0.01562 & $1.477\times10^{-3}$ & $4.619\times10^{-3}$ \\
128 & 0.00781 & $6.894\times10^{-4}$ & $2.746\times10^{-3}$ \\
\bottomrule
\end{tabular}
\caption{Mean Monte Carlo error ($M=24$ realizations), measured as the
maximum over the mesh points, as a function of $N$ ($h=T/N$).}
\label{tab:errores}
\end{table}

The log-log regression slopes are approximately
\[
\text{backward-Euler}:\ 1.18,\qquad
\text{explicit Euler}:\ 0.84.
\]
For $\rho=0.02$ and $\beta=0.35$, Theorem~\ref{thm:convergencia-backward}
guarantees the exponent $H-\rho-\beta=0.33$, while the limiting value of
Corollary~\ref{cor:2H-1} is $2H-1=0.4$. The empirical slopes over the
mesh range considered are higher than these exponents. This does not
contradict the theoretical result: the theorem provides a guaranteed
rate under the general hypotheses and does not assert that
$H-\rho-\beta$ is the exact asymptotic rate for every choice of
coefficients.

The quantitative interpretation of the slopes should be made with
caution. The exact solution of \eqref{eq:example-sfde} is not
available, and the numerical reference is a backward-Euler
approximation computed on a finer mesh. Hence, the comparison may
partially favour the scheme from the same family, and the observed
slopes may reflect both pre-asymptotic effects and the error of the
reference itself. The results nevertheless show a regular decay of the
error under mesh refinement for both methods. Over the range
considered, the backward-Euler scheme exhibits a smaller discrepancy
with respect to the numerical reference than the explicit Euler scheme.

\begin{figure}[ht]
\centering
\includegraphics[width=\textwidth]{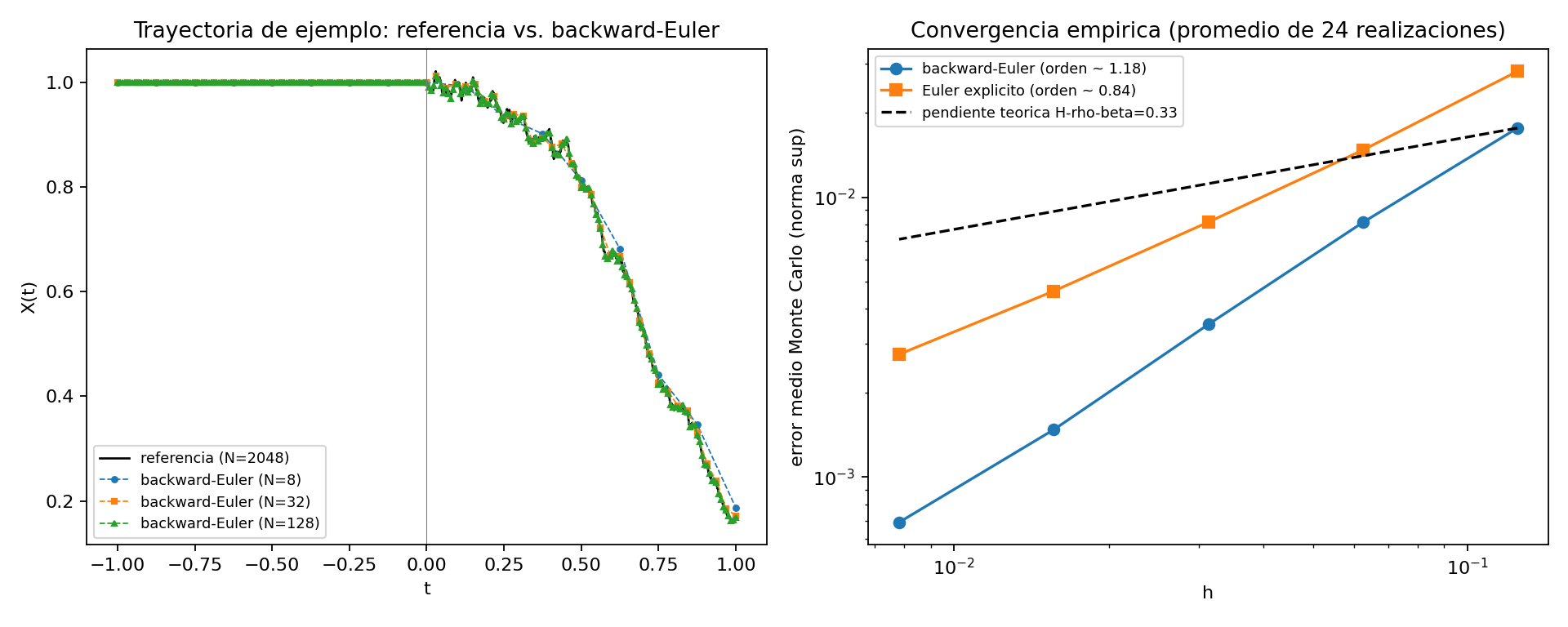}
\caption{Left: reference path ($N=2048$) together with the
backward-Euler approximation for $N=8,32,128$. Right: empirical
convergence in the sup norm (average over 24 realizations) for both
schemes; the guaranteed rate $H-\rho-\beta$ is included for
reference.}
\label{fig:resultados}
\end{figure}
\FloatBarrier

\appendix
\section{Technical Details for the Error Transfer}
\label{app:transfer-details}

This appendix records the singular integrations used in
Lemmas~\ref{lem:delay-error-backward} and
\ref{lem:diffusion-transfer-backward}. Its purpose is to make explicit
which estimates depend only on the left-point projection $C_h$ and are
therefore unchanged by the backward evaluation of the drift.

\subsection{Fractional integral estimate}
For a path $f$ for which the right-hand side is finite, the generalized
Lebesgue-Stieltjes estimate used throughout the paper has the form
\begin{equation}
\label{eq:appendix-zahle-bound}
\left|\int_a^b f(s)\,dB^H(s)\right|
\le C(\omega)\left\{
\int_a^b\frac{|f(s)|}{(s-a)^\beta}\,ds
+\int_a^b\int_a^s
\frac{|f(s)-f(v)|}{(s-v)^{\beta+1}}\,dv\,ds
\right\}.
\end{equation}
The constant is finite on every trajectory for which
$[B^H]_{H-\rho;[0,T]}<\infty$, because
$1-H<\beta<H-\rho$.

\subsection{Seven-term decomposition of the delay increment}
\label{app:delay-seven-terms}
For $0\le v<u\le T$, the second-increment estimate for $K$ applied in
the proof of Lemma~\ref{lem:delay-error-backward} gives
\[
\int_0^u\frac{\Delta_h^Y(v,u)}{(u-v)^{\beta+1}}\,dv
\le C\sum_{j=1}^7M_j(u),
\]
with
\begin{align*}
M_1(u)&=\int_0^u\!\int_{-r}^0
 |s-C_h(s)|(u-v)^{-\beta}\,ds\,dv,\\
M_2(u)&=\int_0^u\!\int_{-r}^0
 |X(v+s)-\widetilde X_h(v+C_h(s))|(u-v)^{-\beta}\,ds\,dv,\\
M_3(u)&=\int_0^u\!\int_{-r}^0
 |X(u+s)-\widetilde X_h(u+C_h(s))|(u-v)^{-\beta}\,ds\,dv,\\
M_4(u)&=\int_0^u\!\int_{-r}^0
 \Upsilon_h(u,v,s)(u-v)^{-\beta-1}\,ds\,dv,\\
M_5(u)&=\int_0^u\!\int_{-r}^0
 |X(u+s)-X(v+s)|\,|s-C_h(s)|
 (u-v)^{-\beta-1}\,ds\,dv,\\
M_6(u)&=\int_0^u\!\int_{-r}^0
 |X(u+s)-X(v+s)|
 |X(v+s)-\widetilde X_h(v+C_h(s))|
 (u-v)^{-\beta-1}\,ds\,dv,\\
M_7(u)&=\int_0^u\!\int_{-r}^0
 |X(u+s)-X(v+s)|
 |X(u+s)-\widetilde X_h(u+C_h(s))|
 (u-v)^{-\beta-1}\,ds\,dv,
\end{align*}
where
\[
\Upsilon_h(u,v,s)=
|X(u+s)-X(v+s)-\widetilde X_h(u+C_h(s))
 +\widetilde X_h(v+C_h(s))|.
\]
This is an algebraic consequence of Lemma~\ref{lem:K-second-increment};
no estimate of the backward drift is used at this stage.

We record the integrations needed in the main proof. Since
$|s-C_h(s)|\le h$,
\[
\int_0^tM_1(u)\,du\le Ch,
\qquad
\int_0^tM_5(u)\,du\le C(\omega)h.
\]
For $M_2$ and $M_3$, insert
\[
|X(q+s)-\widetilde X_h(q+C_h(s))|
\le |E_h(q+s)|
 +|\widetilde X_h(q+s)-\widetilde X_h(q+C_h(s))|,
\]
with $q=v$ or $q=u$. The first term vanishes when $q+s\le0$ and is
bounded by $Z_h(q)$ otherwise; the second is at most
$C(\omega)h^{H-\rho}$. Since $\beta<1$,
\[
\int_0^t\{M_2(u)+M_3(u)\}\,du
\le C\int_0^tZ_h(u)\,du+C(\omega)h^{H-\rho}.
\]
Using additionally the $H-\rho$ Hölder regularity of $X$ in the
product terms gives in the same way
\[
\int_0^t\{M_6(u)+M_7(u)\}\,du
\le C\int_0^tZ_h(u)\,du+C(\omega)h^{H-\rho}.
\]

It remains to control $M_4$. Write $M_4(u)=\sum_{j=1}^6A_j(u)$, where
\begin{align*}
A_1(u)&=\int_0^{C_h(u)}\int_{-r}^{-u}
 \Upsilon_h(u,v,s)(u-v)^{-\beta-1}\,ds\,dv,\\
A_2(u)&=\int_{C_h(u)}^{u}\int_{-r}^{-u}
 \Upsilon_h(u,v,s)(u-v)^{-\beta-1}\,ds\,dv,\\
A_3(u)&=\int_0^{C_h(u)}\int_{-u}^{-C_h(u)}
 \Upsilon_h(u,v,s)(u-v)^{-\beta-1}\,ds\,dv,\\
A_4(u)&=\int_{C_h(u)}^{u}\int_{-u}^{-C_h(u)}
 \Upsilon_h(u,v,s)(u-v)^{-\beta-1}\,ds\,dv,\\
A_5(u)&=\int_0^{C_h(u)}\int_{-C_h(u)}^{0}
 \Upsilon_h(u,v,s)(u-v)^{-\beta-1}\,ds\,dv,\\
A_6(u)&=\int_{C_h(u)}^{u}\int_{-C_h(u)}^{0}
 \Upsilon_h(u,v,s)(u-v)^{-\beta-1}\,ds\,dv.
\end{align*}
The first four regions contain only initial-history terms or a crossing
of the origin. Hölder continuity of $\phi$, of $X$, and of
$\widetilde X_h$, together with $u-C_h(u)\le h$, gives
\begin{align}
\label{eq:A1-A4-delay}
\int_0^t\{A_1(u)+A_2(u)+A_3(u)+A_4(u)\}\,du
\le C(\omega)h^{H-\rho-\beta}.
\end{align}
For completeness, the potentially most singular factors are
$h^{H-\rho}(u-C_h(u))^{-\beta}$ on $A_1$ and
$(u-C_h(u))^{H-\rho-\beta}$ on the short $v$-intervals; their
integrals are of order $h^{H-\rho-\beta}$ because
$H-\rho-\beta>0$.

For $A_5$, split the memory interval into
$[-C_h(u),-v]$, $[-v,-C_h(v)]$, and $[-C_h(v),0]$.
On the first part the error is bounded by
$Z_h(u)+C(\omega)h^{H-\rho}$; on the second part all four arguments
are within distance $u-v$ of the origin and Hölder regularity applies;
on the third part,
\begin{align*}
\Upsilon_h(u,v,s)
&\le \Delta_h^X(v+s,u+s)\\
&\quad+|\widetilde X_h(u+s)-\widetilde X_h(u+C_h(s))|\\
&\quad+|\widetilde X_h(v+s)-\widetilde X_h(v+C_h(s))|.
\end{align*}
The first term produces $\kappa_h(u)$ by its definition, and the last
two are bounded by $C(\omega)h^{H-\rho}$. Consequently,
\begin{equation}
\label{eq:A5-delay}
\int_0^tA_5(u)\,du
\le C\int_0^tZ_h(u)\,du
 +C\int_0^t\kappa_h(u)\,du
 +C(\omega)h^{H-\rho-\beta}.
\end{equation}
Finally, on $A_6$ both $u+s$ and $v+s$ are nonnegative, and the
$H-\rho$ Hölder estimates give
$\Upsilon_h(u,v,s)\le C(\omega)(u-v)^{H-\rho}$; the $v$-interval has
length at most $h$, hence
\begin{equation}
\label{eq:A6-delay}
\int_0^tA_6(u)\,du\le C(\omega)h^{H-\rho-\beta}.
\end{equation}
Combining \eqref{eq:A1-A4-delay}-\eqref{eq:A6-delay} proves
\eqref{eq:delay-M4-bound}, and therefore the delay increment estimate
used in the main convergence proof.

\subsection{Diffusion transfer through the left-point projection}
With $G_h$ as in Lemma~\ref{lem:diffusion-transfer-backward}, write
\[
G_h=G_h^{(1)}+G_h^{(2)}+G_h^{(3)},
\]
where
\begin{align*}
G_h^{(1)}(s)&=\sigma(s,X(s),Y(s))
 -\sigma(s,\widetilde X_h(s),\widetilde Y(s)),\\
G_h^{(2)}(s)&=\sigma(s,\widetilde X_h(s),\widetilde Y(s))
 -\sigma(C_h(s),\widetilde X_h(s),\widetilde Y(s)),\\
G_h^{(3)}(s)&=\sigma(C_h(s),\widetilde X_h(s),\widetilde Y(s))
 -\sigma(C_h(s),\widetilde X_h(C_h(s)),
                 \widetilde Y(C_h(s))).
\end{align*}
This decomposition is useful because it separates the true state error,
the time-projection error, and the within-cell interpolation error. In
particular, it contains only $C_h$ and is identical for the explicit
and backward schemes.

For the first part, Hypothesis~1 and
\eqref{eq:delay-point-error-backward} give
\[
|G_h^{(1)}(s)|\le C Z_h(s)+C(\omega)h^{H-\rho}.
\]
For $0\le v<s$, Lemma~\ref{lem:sigma-second-increment}, whose
finite-dimensional assumptions follow from the functional
Hypothesis~2 by Lemma~\ref{lem:sigma-restriction}, is applied to the
four pairs
\[
(X(s),Y(s)),\quad(\widetilde X_h(s),\widetilde Y(s)),\quad
(X(v),Y(v)),\quad(\widetilde X_h(v),\widetilde Y(v)).
\]
It yields, after using the pathwise a priori bounds,
terms of the form
\[
\Delta_h^X(v,s)+\Delta_h^Y(v,s),
\qquad
\{Z_h(s)+Z_h(v)+h^{H-\rho}\}|s-v|^{H-\rho},
\]
plus products with a bounded factor. Substitution in
\eqref{eq:appendix-zahle-bound} produces the $Z_h$, $\theta_h$ and
delay-increment contributions displayed in
\eqref{eq:diffusion-sup-transfer}.

For $G_h^{(2)}$, the time-Lipschitz bound in Hypothesis~1 gives
$|G_h^{(2)}(s)|\le L_1h$. In the fractional increment integral we split
$v\in[0,s]$ into $[0,C_h(s)]$ and $[C_h(s),s]$. On the first region
both time projection errors are individually of order $h$; on the
second region the integration interval has length at most $h$ and the
Hölder bounds for $\widetilde X_h$ and $\widetilde Y$ are used. Thus
\begin{equation}
\label{eq:G2-fractional}
\|G_h^{(2)}\|_{\beta,1;[0,t]}
\le C(\omega)\{h+h^{1-\beta}+h^{H-\rho-\beta}\}
\le C(\omega)h^{H-\rho-\beta},
\end{equation}
where $\|\cdot\|_{\beta,1}$ denotes the two terms on the right-hand
side of \eqref{eq:appendix-zahle-bound}.

For $G_h^{(3)}$, Hypothesis~1, the $H-\rho$ Hölder regularity of the
interpolation and the analogous estimate for $\widetilde Y$ give
$|G_h^{(3)}(s)|\le C(\omega)h^{H-\rho}$. The same-cell/different-cell
split as above yields
\begin{equation}
\label{eq:G3-fractional}
\|G_h^{(3)}\|_{\beta,1;[0,t]}
\le C(\omega)h^{H-\rho-\beta}.
\end{equation}
Combining the three pieces with
\eqref{eq:appendix-zahle-bound} proves
\eqref{eq:diffusion-sup-transfer}.

For the increment estimate on $[v,u]$, apply the same three-part
decomposition with the lower endpoint $v$ in
\eqref{eq:appendix-zahle-bound}, and then integrate the result against
$(u-v)^{-\beta-1}dv$. Fubini's theorem gives the two Volterra kernels
\[
\int_0^u\mathcal E_h(s)(u-s)^{-\beta}\,ds,
\qquad
\int_0^u Z_h(s)(u-s)^{-2\beta}\,ds,
\]
while \eqref{eq:G2-fractional}-\eqref{eq:G3-fractional} contribute
$C(\omega)h^{H-\rho-\beta}$. This proves
\eqref{eq:diffusion-theta-transfer}. Repeating the same calculation on
the shifted intervals $[v+s,u+s]$ and integrating over
$-C_h(v)\le s\le0$ gives
\eqref{eq:diffusion-kappa-transfer}. No occurrence of $C_h^+$ is
created in any of these steps; this is the precise transfer property
used in the backward analysis.
\vspace{0.5cm}

\bibliographystyle{plain}
\bibliography{references}
\end{document}